\documentclass[12pt]{amsart}

\usepackage[T1]{fontenc}
\usepackage[utf8]{inputenc}
\usepackage{amsmath,amssymb,amsthm,amsfonts,mathtools}
\usepackage{tikz}
\usetikzlibrary{arrows.meta,calc,positioning,decorations.pathreplacing,backgrounds}
\usepackage{booktabs,longtable,array}
\usepackage[shortlabels]{enumitem}
\usepackage{hyperref}
\hypersetup{colorlinks=true,linkcolor=blue,citecolor=blue,urlcolor=blue}
\usepackage{xcolor}

\newenvironment{ack}{\section*{Acknowledgments}\noindent}{}
\usepackage[margin=1in]{geometry}

\definecolor{mynavy}{RGB}{0,31,63}
\definecolor{mygold}{RGB}{215,163,9}
\definecolor{myteal}{RGB}{0,128,128}
\definecolor{mycrimson}{RGB}{153,0,0}

\numberwithin{equation}{section}

\theoremstyle{plain}
\newtheorem{theorem}{Theorem}[section]
\newtheorem{lemma}[theorem]{Lemma}
\newtheorem{proposition}[theorem]{Proposition}
\newtheorem{corollary}[theorem]{Corollary}

\theoremstyle{definition}
\newtheorem{definition}[theorem]{Definition}
\newtheorem{example}[theorem]{Example}
\newtheorem{remark}[theorem]{Remark}

\DeclareMathOperator{\Idem}{Idem}
\DeclareMathOperator{\Hom}{Hom}
\DeclareMathOperator{\Gal}{Gal}
\DeclareMathOperator{\Jac}{J}
\DeclareMathOperator{\rad}{rad}

\DeclareMathOperator{\ch}{char}

\DeclareMathOperator{\wt}{wt}
\DeclareMathOperator{\Cay}{Cay}

\newcommand{\Z}{\mathbb{Z}}
\newcommand{\Q}{\mathbb{Q}}
\newcommand{\R}{\mathbb{R}}
\newcommand{\C}{\mathbb{C}}
\newcommand{\F}{\mathbb{F}}
\newcommand{\calA}{\mathcal{A}}
\newcommand{\calO}{\mathcal{O}}
\newcommand{\Ghat}{\widehat{G}}

\begin{document}


\title{Spectral Algebras of Abelian Cayley Graphs}

\author{Deep Bhattacharjee}
\address{Electro-Gravitational Space Propulsion Laboratory (EGSPL),
  Bhubaneswar, Odisha 751030, India}
\email{itsdeep@live.com; d.bhattacharjee@erl-forschung.de}

\author{Priyabrata Mandal}
\address{Department of Mathematics, Bioinformatics and Computer Applications,
  Maulana Azad National Institute of Technology Bhopal, Bhopal 462003, India}
\email{priyabrata@manit.ac.in}
\thanks{Priyabrata Mandal is the corresponding author}

\author{Ushashi Bhattacharya}
\address{National Taiwan University, No.~1, Section~4, Roosevelt Road,
  Zhongzheng District, Taipei City 10617, Taiwan}
\email{rai.bhattacharya.in@gmail.com}

\date{\today}

\begin{abstract}
Let $G$ be a finite abelian group of order $N$, $S \subseteq G
\setminus \{0\}$ a symmetric connection set, and $K$ a field with
$\ch(K) \nmid N$.
The \emph{spectral algebra} $\calA_K(\Cay(G,S)) = K[A]$ generated by
the adjacency matrix of the Cayley graph is proved to decompose, via the
\emph{character-orbit decomposition}, as a semisimple product of field
extensions of $K$, one factor for each $\Gal(\overline{K}/K)$-orbit of the
eigenvalues $\lambda_\chi = \sum_{s \in S} \chi(s)$.
The proof uses the abelian discrete Fourier transform to diagonalise $A$,
the Galois action on the character group $\Ghat$ to partition eigenvalues
into orbits, and the Chinese Remainder Theorem to convert the squarefree
minimal polynomial into a Wedderburn product.
The dimension of $\calA_K(\Cay(G,S))$ equals the number of distinct
eigenvalues, the idempotent count is $2^r$ where $r$ is the orbit number,
and primitive idempotents are computed explicitly via the Bezout algorithm
in $K[x]$.
Over $\Q$, every Wedderburn summand is a real subfield of the cyclotomic
field $\Q(\zeta_N)$.
New results include: a tensor-product comparison for Cartesian products of
Cayley graphs; a systematic analysis of the spectral algebra for
elementary abelian groups $(\Z/p)^k$ (rational for $p \le 3$, requiring
real cyclotomic extensions for $p \ge 5$); and a worked orbit analysis for
non-cyclic groups including $\Z/6 \times \Z/2$ and $\Z/5 \times \Z/2$.
The cyclic case recovers the companion result
$\calA_\Q(C_n) \cong \prod_{d \mid n} \Q(\zeta_d)^+$; the Hamming cube
gives $\calA_\Q(\Q_k) \cong \Q^{k+1}$.
The characteristic-$p$ case is also treated.
\end{abstract}


\maketitle
\markleft{Bhattacharjee, Mandal, and Bhattacharya}
\setcounter{tocdepth}{2}

\section{Introduction}
\label{sec:intro}

\subsection{Background}

The adjacency matrix $A$ of a graph $\Gamma$ on $n$ vertices generates a
commutative unital $K$-subalgebra $\calA_K(\Gamma) = K[A]$ of the full
matrix ring $M_n(K)$.  For a Cayley graph $\Cay(G, S)$ on a finite abelian
group $G$ of order $N$, the spectrum is governed by the character theory of
$G$: the eigenvalues are character sums $\lambda_\chi = \sum_{s \in S}
\chi(s)$, one for each character $\chi \in \Ghat$~\cite{Ter99,BH12}.
The symmetry $S = -S$
forces $\lambda_\chi = \lambda_{\chi^{-1}}$ and, over $\Q$, ensures
$\lambda_\chi \in \Q(\zeta_N)^+$ for every $\chi$.

This paper answers the question: \emph{what does $\calA_K(\Cay(G,S))$
look like as an abstract $K$-algebra?}  The answer, the
\emph{character-orbit decomposition}, is precisely the isomorphism
in~\eqref{eq:main-intro}: a product of field extensions, one per
$\Gal(\overline{K}/K)$-orbit of eigenvalues.

The decomposition has the following ingredients.
The DFT diagonalisation (Section~\ref{sec:dft}) shows $A$ is diagonalisable
over $\overline{K}$ whenever $\ch(K) \nmid N$, so the minimal polynomial
$m_{G,S,K}$ is squarefree.  The Galois action on the character group
(Section~\ref{sec:orbits}) partitions the eigenvalues into orbits, giving
a factorisation of $m_{G,S,K}$ into pairwise coprime irreducibles.  The
Chinese Remainder Theorem (Section~\ref{sec:proof}) converts this
factorisation into a Wedderburn product.

The paper recovers known special cases: for $G = \Z/n$ and $S = \{1,-1\}$,
one gets $\calA_\Q(C_n) \cong \prod_{d \mid n} \Q(\zeta_d)^+$, the
cyclotomic decomposition of the cycle graph established in the companion
paper~\cite{MKB26}; for the Hamming cube $Q_k$, one gets
$\calA_\Q(Q_k) \cong \Q^{k+1}$ (rational, since all eigenvalues are
integers).  New results include: a tensor product comparison for Cartesian
products (Section~\ref{sec:products}), a partial classification for
elementary abelian groups $(\Z/p)^k$ for all primes $p$
(Section~\ref{sec:elem}), the Bose--Mesner algebra connection
(Section~\ref{sec:bose}), coding-theory applications
(Section~\ref{sec:coding}), and the Galois module structure of each
Wedderburn summand (Section~\ref{sec:galois-structure}).

The main result of this paper is the \emph{character-orbit decomposition}:
\begin{equation}
\label{eq:main-intro}
  \calA_K\!\bigl(\Cay(G,S)\bigr)
  \;\cong\;
  \prod_{\calO\;\in\;\calO_K(G,S)}
  K[x]\bigl/\bigl(\Psi_{\calO,K}(x)\bigr),
\end{equation}
where $\calO_K(G,S)$ is the partition of eigenvalues into
$\Gal(\overline{K}/K)$-orbits and $\Psi_{\calO,K}$ is the minimal
polynomial over $K$ of any element of $\calO$.  Each factor in the product
is a field; the algebra is always semisimple when $\ch(K) \nmid N$.

The spectral theory of abelian Cayley graphs has been studied extensively:
eigenvalues as character sums appear in coding theory, combinatorics, and
number theory~\cite{BH12,CDS79}.  The cyclotomic fields $\Q(\zeta_d)^+$
are central to algebraic number theory~\cite{IR90,Neu99,Was97}.  Wedderburn
decompositions of commutative semisimple algebras are classical~\cite{DF04,
Jac80,Lang02}.  Homomorphism counts for cyclic rings were studied in~\cite{GVB84}.
The contribution here is the explicit identification of
$\calA_K(\Cay(G,S))$ with a product of fields named by the Galois orbit
structure, together with new results on direct products, elementary abelian
groups, and the Bose--Mesner algebra connection.

\textit{MSC 2020.}  Primary: 05C25 (graphs and groups), 11R18 (cyclotomic
extensions).  Secondary: 16S99 (associative rings NEC), 11A25 (arithmetic
functions), 13B05 (polynomial rings).

\subsection{Organisation}

Section~\ref{sec:cayley} sets up Cayley graphs and character theory.
Section~\ref{sec:specalg} defines spectral algebras.  Section~\ref{sec:dft}
establishes the abelian DFT eigenvalue formula.  Section~\ref{sec:orbits}
analyses the Galois action.  Section~\ref{sec:proof} proves the main
theorem.  Section~\ref{sec:struct} gives structural corollaries.
Section~\ref{sec:idem} constructs primitive idempotents.
Section~\ref{sec:cyclic} specialises to cyclic groups.
Section~\ref{sec:products} treats Cartesian products of Cayley graphs.
Section~\ref{sec:elem} analyses elementary abelian groups.
Section~\ref{sec:bose} connects to Bose--Mesner algebras.
Section~\ref{sec:examples} works further examples.
Section~\ref{sec:charp} handles characteristic~$p$.
Section~\ref{sec:conc} concludes.

\section{Cayley Graphs and Characters}
\label{sec:cayley}

\begin{definition}[Cayley graph]
Let $G$ be a finite abelian group and $S \subseteq G \setminus \{0\}$ with
$S = -S$.  The \emph{Cayley graph} $\Cay(G, S)$ has vertex set $G$ and
undirected edge set $\{\{g, g+s\} : g \in G,\, s \in S\}$.  It is
$|S|$-regular and connected if and only if $S$ generates $G$.
\end{definition}

Every undirected Cayley graph on $G = \Z/n$ with symmetric connection set
is a \emph{circulant graph}.  The adjacency matrix $A \in M_N(\Z)$ has
$A_{g,h} = \mathbf{1}[h - g \in S]$.

\subsection{Character sums and the eigenvalue formula}

The \emph{character group} $\Ghat = \Hom(G, \C^*)$ of $G$ is isomorphic
(though not canonically) to $G$.  Each character satisfies $\chi(g)^N = 1$,
so character values are $N$-th roots of unity.

For a symmetric Cayley graph ($S = -S$), the eigenvalue of $A$ at
character $\chi$ is
\begin{equation}
\label{eq:eigenvalue}
  \lambda_\chi \coloneqq \sum_{s \in S} \chi(s).
\end{equation}
The symmetry $S = -S$ forces $\lambda_\chi = \lambda_{\chi^{-1}}$:
$\lambda_{\chi^{-1}} = \sum_{s \in S} \chi^{-1}(s) = \sum_{s \in S}
\chi(-s) = \sum_{s' \in -S} \chi(s') = \sum_{s \in S} \chi(s) = \lambda_\chi$.
Over $\C$, the same identity says $\lambda_\chi = \overline{\lambda_\chi}$,
i.e., $\lambda_\chi \in \R$.

\begin{lemma}[Character sum properties]
\label{lem:char-sum}
For $\Cay(G,S)$ undirected: $(i)$ $\lambda_\chi = \lambda_{\chi^{-1}}$;
$(ii)$ $\lambda_{\chi_0} = |S|$ where $\chi_0$ is the principal character;
$(iii)$ $|\lambda_\chi| \le |S|$.
\end{lemma}

\begin{proof}
(i) was proved above.  (ii) is immediate since $\chi_0(s) = 1$ for all $s$.
(iii) follows from $|\chi(s)| = 1$ and the triangle inequality.
\end{proof}

\section{The Spectral Algebra}
\label{sec:specalg}

\begin{definition}[Spectral algebra]
\label{def:spec}
For a graph $\Gamma$ on $n$ vertices with adjacency matrix $A \in M_n(\Z)$
and a field $K$, the \emph{spectral algebra} is
\[
  \calA_K(\Gamma) \coloneqq K[A]
  = \{f(A) : f \in K[x]\} \subseteq M_n(K).
\]
It is a commutative unital $K$-subalgebra of $M_n(K)$.
\end{definition}

\begin{proposition}[Quotient-ring presentation]
\label{prop:quotient}
Let $m = m_{\Gamma,K} \in K[x]$ be the minimal polynomial of $A$ over $K$.
The evaluation map $\mathrm{ev}_A : f \mapsto f(A)$ induces a $K$-algebra
isomorphism
\[
  \calA_K(\Gamma) \;\cong\; K[x]/(m_{\Gamma,K}),
\]
and $\dim_K \calA_K(\Gamma) = \deg m_{\Gamma,K}$.
\end{proposition}

\begin{proof}
The map $\mathrm{ev}_A : K[x] \to \calA_K(\Gamma)$ is surjective with
kernel $(m_{\Gamma,K})$.  The first isomorphism theorem gives the result;
the $K$-basis $\{1, \bar{x}, \ldots, \bar{x}^{d-1}\}$ of $K[x]/(m)$
shows $\dim_K \calA_K(\Gamma) = \deg m$.
\end{proof}

\begin{proposition}[Semisimplicity criterion]
\label{prop:ss}
$\calA_K(\Gamma)$ is semisimple if and only if $m_{\Gamma,K}$ is squarefree.
\end{proposition}

\begin{proof}
$\Jac(K[x]/(m)) = (\rad(m))/(m)$, so $\Jac = 0$ iff $m$ is squarefree.
\end{proof}

\section{Abelian DFT and Eigenvalues}
\label{sec:dft}

\subsection{DFT diagonalisation}

\begin{theorem}[Abelian DFT diagonalisation]
\label{thm:dft}
Let $G = \{g_0, \ldots, g_{N-1}\}$, $\Ghat = \{\chi_0, \ldots, \chi_{N-1}\}$,
and $K$ a field with $\ch(K) \nmid N$.

Since $\ch(K) \nmid N$, the polynomial $x^N - 1$ is separable over $K$
(its derivative $N x^{N-1}$ is non-zero and coprime to $x^N - 1$).
The algebraic closure $\overline{K}$ therefore contains exactly $N$ distinct
$N$-th roots of unity; fix a primitive one, $\zeta_N \in \overline{K}$.
Regard each character $\chi_j \in \Ghat$ as a group homomorphism
$G \to \overline{K}^*$ whose values are $N$-th roots of unity in $\overline{K}$,
and define the \emph{DFT matrix}
$F = (\chi_j(g_i))_{0 \le i,j < N} \in M_N(\overline{K})$.

Then $F$ is invertible with
$F^{-1} = N^{-1}(\chi_j(-g_i))^\top$, and
\begin{equation}
\label{eq:Fdiag}
  F^{-1} A F = \mathrm{diag}(\lambda_{\chi_0}, \ldots, \lambda_{\chi_{N-1}})
  \quad \text{in } M_N(\overline{K}),
\end{equation}
where $\lambda_\chi = \sum_{s \in S} \chi(s)$.
Consequently, $A$ is diagonalisable over $\overline{K}$, and its minimal
polynomial $m_{G,S,K} \in K[x]$ is squarefree over $K$: since $A$ has $N$
distinct eigenvalues in $\overline{K}$, no irreducible factor of $m_{G,S,K}$
over $K$ can occur with multiplicity greater than one.
\end{theorem}

\begin{proof}
\emph{Invertibility.}  The $(j,j')$-entry of $F \cdot N^{-1} (\chi_j(-g_i))^\top$
is $N^{-1} \sum_{g \in G} (\chi_j \chi_{j'}^{-1})(g)$.
The character $\chi_j \chi_{j'}^{-1}$ is principal iff $j = j'$.
For $j \ne j'$, pick $h \in G$ with $(\chi_j \chi_{j'}^{-1})(h) \ne 1$
(such $h$ exists since $\chi_j \ne \chi_{j'}$); substituting $g \mapsto h+g$
gives $(\chi_j \chi_{j'}^{-1})(h) \cdot \sum_g (\chi_j \chi_{j'}^{-1})(g)
= \sum_g (\chi_j \chi_{j'}^{-1})(g)$,
forcing the sum to zero, since
$(\chi_j \chi_{j'}^{-1})(h) - 1 \ne 0$ in $\overline{K}$
(the $N$-th roots of unity are distinct because $\ch(K) \nmid N$).
Hence $F \cdot N^{-1}(\chi_j(-g_i))^\top = I_N$.

\emph{Eigenvalue formula.}  The $j$-th column $v_j = (\chi_j(g_0), \ldots,
\chi_j(g_{N-1}))^\top$ satisfies
$(A v_j)_i = \sum_{s \in S} \chi_j(g_i + s)
= \chi_j(g_i) \sum_{s \in S} \chi_j(s)
= \lambda_{\chi_j}(v_j)_i$,
so $v_j$ is an eigenvector of $A$ with eigenvalue $\lambda_{\chi_j}$.
\end{proof}

\subsection{Cyclotomic containment}

\begin{corollary}[Cyclotomic containment]
\label{cor:cyclo}
For $K = \Q$ and $S = -S$, every eigenvalue $\lambda_\chi$ lies in the
maximal real subfield $\Q(\zeta_N)^+ = \Q(\zeta_N + \zeta_N^{-1})$.
In particular, $[\Q(\lambda_\chi) : \Q]$ divides $\varphi(N)/2$.
\end{corollary}

\begin{proof}
Each $\chi(s)$ is an $N$-th root of unity, so $\lambda_\chi \in \Q(\zeta_N)$.
The symmetry $\lambda_\chi = \lambda_{\chi^{-1}}$ and the fact that
$\chi^{-1}$ corresponds to complex conjugation on $\Q(\zeta_N)$ give
$\lambda_\chi \in \R$, hence $\lambda_\chi \in \Q(\zeta_N) \cap \R =
\Q(\zeta_N)^+$.
\end{proof}

The character-orbit decomposition and the commutative diagram underlying
it are shown schematically in Figure~\ref{fig:cd}.

\begin{figure*}[ht]
\centering
\begin{tikzpicture}[
  node distance=2.8cm,
  every node/.style={font=\small},
  arr/.style={->,>=Stealth,thick}]
\node (A) at (0,2.5) {$K[x]$};
\node (B) at (5.5,2.5) {$\calA_K(\Cay(G,S))$};
\node (C) at (0,0) {$K[x]/(m_{G,S,K})$};
\node (D) at (5.5,0) {$\prod_{\calO \in \calO_K(G,S)} K[x]/(\Psi_{\calO,K})$};
\draw[arr] (A) -- node[above] {$\mathrm{ev}_A:\,f \mapsto f(A)$} (B);
\draw[arr] (A) -- node[left] {$\pi$} (C);
\draw[arr,double,double distance=1.5pt] (B) -- node[right] {$\overline{\Phi}$} (D);
\draw[arr,double,double distance=1.5pt] (C) -- node[below] {$\mathrm{C.R.T.}$} (D);
\end{tikzpicture}
\caption{The commutative diagram for the character-orbit decomposition.
  The evaluation map at the top is surjective with kernel $(m_{G,S,K})$.
  The Chinese Remainder Theorem (bottom arrow) applies because
  $m_{G,S,K} = \prod_{\calO} \Psi_{\calO,K}$ is a squarefree product of
  pairwise coprime irreducibles.  Commutativity follows from
  Proposition~\ref{prop:quotient}.}
\label{fig:cd}
\end{figure*}

\section{Galois Orbits on Character Sums}
\label{sec:orbits}

\begin{lemma}[Galois action on eigenvalues]
\label{lem:galois-action}
For $K = \Q$, the group $\Gal(\Q(\zeta_N)/\Q) \cong (\Z/N\Z)^*$ acts on
the eigenvalue set by $\sigma_a(\lambda_\chi) = \lambda_{\chi^a}$, where
$\chi^a(g) = \chi(g)^a$.
\end{lemma}

\begin{proof}
$\sigma_a(\zeta_N) = \zeta_N^a$ implies $\sigma_a(\chi(g)) = \chi(g)^a$
for each character value.  Summing over $S$:
$\sigma_a(\lambda_\chi) = \sum_{s \in S} \chi(s)^a
= \sum_{s \in S} \chi^a(s) = \lambda_{\chi^a}$.
\end{proof}

\begin{remark}
The identity $\lambda_\chi = \lambda_{\chi^{-1}}$ (Lemma~\ref{lem:char-sum})
means the effective Galois group on distinct eigenvalues is
$(\Z/N\Z)^*/\{\pm 1\}$.  The stabiliser of $\lambda_\chi$ in
$(\Z/N\Z)^*$ is $\{a : \lambda_{\chi^a} = \lambda_\chi\}$, and the orbit
size equals the quotient group index.
\end{remark}

\begin{definition}[Orbit partition]
\label{def:orbits}
For $K$ with $\ch(K) \nmid N$: $\calO_K(G,S)$ is the partition of
$\{\lambda_\chi : \chi \in \Ghat\}$ into $\Gal(\overline{K}/K)$-orbits.
For each $\calO \in \calO_K(G,S)$, write $\Psi_{\calO,K}$ for the
monic minimal polynomial over $K$ of any orbit element.  Set
$r = |\calO_K(G,S)|$ for the orbit count.
\end{definition}

\begin{lemma}[Coprimality]
\label{lem:coprime}
For distinct orbits $\calO \ne \calO'$,
$\gcd(\Psi_{\calO,K}, \Psi_{\calO',K}) = 1$ in $K[x]$.
\end{lemma}

\begin{proof}
Distinct monic irreducibles in the PID $K[x]$ are coprime; they are
distinct because $\calO \ne \calO'$ implies disjoint root sets.
\end{proof}

\section{Proof of the Main Theorem}
\label{sec:proof}

\begin{theorem}[Character-orbit decomposition]
\label{thm:main}
Let $G$ be a finite abelian group of order $N$, $S \subseteq G
\setminus \{0\}$ symmetric, and $K$ a field with $\ch(K) \nmid N$.  Then:
\begin{enumerate}[label=\textup{(\roman*)}]
  \item $m_{G,S,K}(x) = \prod_{\calO \in \calO_K(G,S)} \Psi_{\calO,K}(x)$,
        a squarefree product of distinct monic irreducibles over $K$;
  \item there is a $K$-algebra isomorphism
        \[
          \calA_K(\Cay(G,S))
          \;\cong\;
          \prod_{\calO \in \calO_K(G,S)} K[x]/(\Psi_{\calO,K}(x));
        \]
  \item $\calA_K(\Cay(G,S))$ is semisimple;
  \item for $K = \Q$, every factor $\Q[x]/(\Psi_{\calO,\Q})$ is a number
        field contained in $\Q(\zeta_N)^+$.
\end{enumerate}
\end{theorem}

\begin{proof}
\emph{Step 1: squarefreeness.}  Since $\ch(K) \nmid N$, the polynomial
$x^N - 1$ is separable over $K$, and $A$ is diagonalisable over
$\overline{K}$ (Theorem~\ref{thm:dft}).  A diagonalisable matrix has
squarefree minimal polynomial; equivalently, every irreducible factor of
$m_{G,S,K}$ over $K$ is separable (no repeated roots in $\overline{K}$),
and no irreducible can appear with multiplicity $\ge 2$ in $m_{G,S,K}$
without contradicting diagonalisability.

\emph{Step 2: orbit factorisation.}  The eigenvalues of $A$ over
$\overline{K}$ are exactly $\{\lambda_\chi : \chi \in \Ghat\}$.  The root
set of $m_{G,S,K}$ is Galois-stable, so its irreducible factors over $K$
are precisely $\{\Psi_{\calO,K} : \calO \in \calO_K(G,S)\}$, one per orbit.
This gives part~(i).

\emph{Step 3: Wedderburn.}  Proposition~\ref{prop:quotient} gives
\[
\calA_K(\Cay(G,S)) \cong K[x]/(m_{G,S,K})
= K[x]/\!\bigl(\prod_\calO \Psi_{\calO,K}\bigr).
\]
The factors are pairwise coprime
(Lemma~\ref{lem:coprime}), so the Chinese Remainder Theorem gives part~(ii).
Each factor $K[x]/(\Psi_{\calO,K})$ is a field.

\emph{Step 4: semisimplicity and cyclotomic containment.}  A product of
fields has $\Jac = 0$, giving~(iii).  For~(iv), Corollary~\ref{cor:cyclo}
gives $\lambda_\calO \in \Q(\zeta_N)^+$ for every orbit representative, so
$\Q[x]/(\Psi_{\calO,\Q}) \cong \Q(\lambda_\calO) \subseteq \Q(\zeta_N)^+$.
\end{proof}

\section{Structural Consequences}
\label{sec:struct}

\begin{corollary}[Dimension]
\label{cor:dim}
$\dim_K \calA_K(\Cay(G,S))$ equals the number of distinct eigenvalues of
$A$ over $\overline{K}$.
\end{corollary}

\begin{proof}
$\dim = \deg m_{G,S,K} = \sum_{\calO} |\calO|
= $ number of distinct eigenvalues.
\end{proof}

\begin{corollary}[Idempotent count]
\label{cor:idem}
$|\Idem(\calA_K(\Cay(G,S)))| = 2^r$ where $r = |\calO_K(G,S)|$.
\end{corollary}

\begin{proof}
A product of $r$ fields has $2^r$ idempotents: each is a binary $r$-tuple
with entries in $\{0, 1\}$ independently.
\end{proof}

\begin{corollary}[Homomorphisms to $K$]
\label{cor:homs}
The number of $K$-algebra homomorphisms from
$\calA_K(\Cay(G,S))$ to $K$ equals the number of orbits
$\calO$ for which $\Psi_{\calO,K}$ has degree $1$, that is,
the number of eigenvalues already in $K$.
\end{corollary}

\begin{proof}
A $K$-algebra map $\prod L_\calO \to K$ annihilates all but one factor and
embeds that factor into $K$; such an embedding exists iff $L_\calO = K$.
\end{proof}

Figure~\ref{fig:dim-bars} visualises the dimension and orbit structure for
several families of abelian Cayley graphs.

\begin{figure*}[ht]
\centering
\begin{tikzpicture}[scale=1.0]
\foreach \lbl/\nb/\gb/\tb/\tot/\yi in {
  {$C_5$}/1/2/0/3/15,
  {$C_7$}/1/0/3/4/14,
  {$C_{11}$}/1/0/5/6/13,
  {$C_{12}$}/5/2/0/7/12,
  {$C_{15}$}/2/2/4/8/11,
  {$Q_2$}/3/0/0/3/10,
  {$Q_3$}/4/0/0/4/9,
  {$Q_4$}/5/0/0/5/8,
  {$Q_5$}/6/0/0/6/7,
  {$\Z/6{\times}\Z/2$}/7/0/0/7/6,
  {$\Z/5{\times}\Z/2$}/2/4/0/6/5,
  {$\Z/4{\times}\Z/4$}/5/0/0/5/4,
  {$\Z/3{\times}\Z/3$}/3/0/0/3/3,
  {$\Z/5{\times}\Z/5$}/2/4/0/6/2,
  {$C_4{\square}C_6$}/9/0/0/9/1%
}{%
  \pgfmathsetmacro\yb{\yi * 0.67}%
  \pgfmathsetmacro\yt{\yb + 0.50}%
  \pgfmathsetmacro\xna{\nb * 0.88}%
  \pgfmathsetmacro\xga{(\nb + \gb) * 0.88}%
  \pgfmathsetmacro\xta{(\nb + \gb + \tb) * 0.88}%
  \pgfmathsetmacro\ym{\yb + 0.25}%
  \fill[mynavy] (0,\yb) rectangle (\xna,\yt);
  \fill[mygold] (\xna,\yb) rectangle (\xga,\yt);
  \fill[myteal] (\xga,\yb) rectangle (\xta,\yt);
  \draw[black,very thin] (0,\yb) rectangle (\xta,\yt);
  \node[left,font=\scriptsize] at (-0.12,\ym) {\lbl};
  \node[right,font=\scriptsize,mynavy] at (\xta+0.10,\ym) {$\tot$};
}
\draw[->,thin] (0,0) -- (9.5,0) node[right,font=\footnotesize] {$\dim_K\calA_K$};
\foreach \x in {0,2,4,6,8}{
  \pgfmathsetmacro\xx{\x*0.88}
  \draw[thin] (\xx,-0.10) -- (\xx,0.10);
  \node[below,font=\scriptsize] at (\xx,-0.12) {$\x$};
}
\fill[mynavy] (0,-1.30) rectangle (0.70,-0.90);
\node[right,font=\footnotesize] at (0.80,-1.10) {rational orbits ($\Psi$ linear)};
\fill[mygold] (4.80,-1.30) rectangle (5.50,-0.90);
\node[right,font=\footnotesize] at (5.60,-1.10) {quadratic orbits};
\fill[myteal] (9.20,-1.30) rectangle (9.90,-0.90);
\node[right,font=\footnotesize] at (10.00,-1.10) {higher degree};
\end{tikzpicture}
\caption{Dimension and Galois orbit structure of $\calA_K(\Cay(G,S))$ for
  fifteen abelian Cayley graphs.  Bar length equals
  $\dim_K \calA_K = $ number of distinct eigenvalues; colours indicate
  orbit type (navy: rational eigenvalue, gold: quadratic irrational, teal:
  higher-degree extension).  Hamming cubes $Q_k$ have all-rational spectra;
  the cyclic primes $C_p$ for $p \ge 7$ require cubic and higher extensions;
  elementary abelian groups $(\Z/5)^2$ and $(\Z/5)^k$ require quadratic
  extensions of $\Q(\sqrt{5})$.  Notation: $C_4\,\square\,C_6$ denotes the
  Cartesian product (grid graph).}
\label{fig:dim-bars}
\end{figure*}
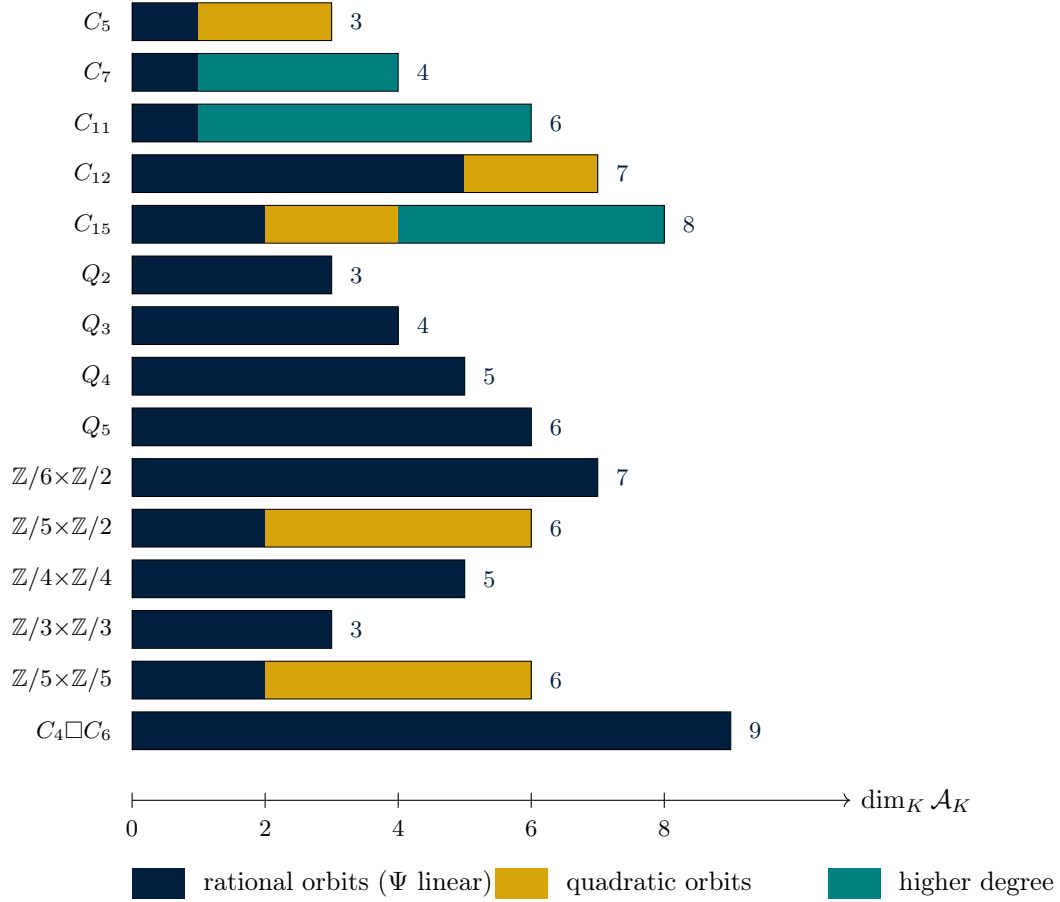

\section{Primitive Idempotents}
\label{sec:idem}

\begin{proposition}[Bezout construction of primitive idempotents]
\label{prop:idem}
For each $\calO \in \calO_K(G,S)$, set $M_\calO = m_{G,S,K}/\Psi_{\calO,K}$.
The extended Euclidean algorithm yields $u_\calO \in K[x]$ with
$M_\calO u_\calO \equiv 1 \pmod{\Psi_{\calO,K}}$.  The element
$e_\calO \coloneqq M_\calO(A) u_\calO(A) \in \calA_K(\Cay(G,S))$
is the unique primitive idempotent satisfying
$e_\calO^2 = e_\calO$, $e_\calO e_{\calO'} = 0$ for $\calO \ne \calO'$,
and $\sum_\calO e_\calO = I$.
The direct summand satisfies $e_\calO \cdot \calA_K(\Cay(G,S)) \cong K[x]/(\Psi_{\calO,K})$.
\end{proposition}

\begin{proof}
Under the $\calO'$-th projection $\pi_{\calO'} : \calA_K \to
K[x]/(\Psi_{\calO',K})$: if $\calO' \ne \calO$ then $\Psi_{\calO',K} \mid
M_\calO$, so $\pi_{\calO'}(e_\calO) = 0$; if $\calO' = \calO$ then the
Bezout relation evaluated at any root of $\Psi_{\calO,K}$ gives
$\pi_\calO(e_\calO) = 1$.  The idempotent relations follow by applying each
projection independently.
\end{proof}

\begin{example}[Hamming cube $Q_2$]
\label{ex:Q2-idem}
$Q_2 = \Cay((\Z/2)^2, \{e_1, e_2\})$ has three distinct eigenvalues:
$2$, $0$, $-2$ with minimal polynomial $(x-2)x(x+2) = x^3 - 4x$.
Complementary products: $M_2 = x(x+2)$, $M_0 = x^2-4$, $M_{-2} = x(x-2)$.
Since $M_2(2) = 8$, $M_0(0) = -4$, $M_{-2}(-2) = 8$, the idempotents are
\begin{align*}
  e_2    &= A(A+2I)/8, \\
  e_0    &= -(A^2-4I)/4, \\
  e_{-2} &= A(A-2I)/8.
\end{align*}
One verifies $e_2 + e_0 + e_{-2} = I$ and $e_i e_j = 0$ for $i \ne j$
from the Cayley--Hamilton identity $A^3 - 4A = 0$.
\end{example}

\begin{example}[Explicit idempotents for $\Cay(\Z/5 \times \Z/2, S)$]
\label{ex:Z5Z2-idem}
From Example~\ref{ex:Z5Z2}, the four Galois orbits have minimal polynomials
$x-3$, $x-1$, $\Psi_A = x^2-x-1$, and $\Psi_B = x^2+3x+1$.
(Orbit $A$ has roots $(\sqrt{5}\pm 1)/2$ with sum $1$ and product $-1$;
orbit $B$ has roots $(\sqrt{5}-3)/2$ and $-(3+\sqrt{5})/2$ with sum $-3$
and product $1$, giving $\Psi_B = x^2+3x+1$.)
The full minimal polynomial is
$m = (x-3)(x-1)(x^2-x-1)(x^2+3x+1)$.

For the orbit $A$ with $\Psi_A = x^2-x-1$, the complementary product is
$M_A(x) = (x-3)(x-1)(x^2+3x+1)$.  In $\Q[x]/(x^2-x-1)$ one has $x^2
\equiv x+1$, so
\begin{align*}
  (x-3)(x-1) &\equiv -3x+4, \\
  x^2+3x+1   &\equiv 4x+2,  \\
  M_A(x) &\equiv (-3x+4)(4x+2) \\
         &= -12x^2+10x+8 \equiv -2x-4.
\end{align*}
The Bezout inverse of $-2(x+2)$ modulo $x^2-x-1$ is $(x-3)/10$
(verified: $-2(x+2)\cdot(x-3)/10 = -(x^2-x-6)/5 \equiv 1$).
The primitive idempotent for orbit $A$ is therefore
\begin{align*}
  e_{\calO_A} &= M_A(A) \cdot (A-3I)/10, \\
  M_A(A) &= (A-3I)(A-I)(A^2+3A+I).
\end{align*}
All entries of this $10 \times 10$ rational matrix are computable from
powers of $A$ up to degree $5$.  The idempotent for orbit $B$ is computed
analogously using $M_B = (x-3)(x-1)(x^2-x-1)$ and its Bezout inverse
modulo $x^2+3x+1$.
\end{example}

\section{The Cyclic Group Case}
\label{sec:cyclic}

When $G = \Z/n$ and $S = \{1, -1\}$, the graph $\Cay(G,S) = C_n$ (the
cycle graph).  The character group $\widehat{\Z/n}$ has characters
$\chi_k(j) = \zeta_n^{kj}$, giving eigenvalues $\lambda_k =
\zeta_n^k + \zeta_n^{-k} = 2\cos(2\pi k/n)$.

\begin{theorem}[Cyclic case]
\label{thm:cyclic}
For $G = \Z/n$, $S = \{1, -1\}$:
\[
  \calA_\Q(C_n) \;\cong\; \prod_{d \mid n} \Q(\zeta_d)^+,
\]
where $\Q(\zeta_d)^+ = \Q(2\cos(2\pi/d))$ and the product runs over all
positive divisors $d$ of $n$.  The dimension is $\lfloor n/2 \rfloor + 1$
and the idempotent count is $2^{\tau(n)}$.
\end{theorem}

\begin{proof}
By Lemma~\ref{lem:galois-action}, $\sigma_a(2\cos(2\pi k/n)) =
2\cos(2\pi ak/n)$.  Two eigenvalues $2\cos(2\pi k/n)$ and
$2\cos(2\pi k'/n)$ are $\Q$-conjugate iff $ak \equiv \pm k' \pmod{n}$
for some $a \in (\Z/n\Z)^*$, which holds iff $\gcd(k,n) = \gcd(k',n)$,
i.e., iff $k$ and $k'$ have the same exact order $d = n/\gcd(k,n)$ for
$\zeta_n^k$.  The orbit indexed by $d \mid n$ is
$\calO_d = \{2\cos(2\pi j/d) : \gcd(j,d) = 1,\, 1 \le j \le \lfloor d/2\rfloor\}$,
with minimal polynomial $\Psi_d(x)$ of degree $\max(1, \varphi(d)/2)$.
The identification $\Q[x]/(\Psi_d) \cong \Q(\zeta_d)^+$ and the dimension
count $\sum_{d \mid n} \max(1,\varphi(d)/2) = \lfloor n/2\rfloor + 1$
complete the proof.  The idempotent count $2^{\tau(n)}$ follows from
$|\calO_\Q(G,S)| = \tau(n)$.
\end{proof}

\begin{example}[Selected cyclic decompositions]
$\calA_\Q(C_{10}) \cong \Q^2 \times \Q(\sqrt{5})^2$; dimension $6$,
idempotents $2^4 = 16$.
$\calA_\Q(C_{12}) \cong \Q^5 \times \Q(\sqrt{3})$; dimension $7$,
idempotents $2^6 = 64$.
$\calA_\Q(C_{30}) \cong \Q^4 \times \Q(\sqrt{5})^2
\times \Q(\zeta_{15})^+ \times \Q(\zeta_{30})^+$; dimension $16$,
idempotents $2^8 = 256$.
\end{example}

\section{Cartesian Products of Cayley Graphs}
\label{sec:products}

The Cartesian product $\Gamma_1 \,\square\, \Gamma_2$ of two graphs has
vertex set $V_1 \times V_2$ and edges
edges $\{(u_1,u_2),(v_1,u_2)\}$
when $u_1 v_1 \in E_1$, and
$\{(u_1,u_2),(u_1,v_2)\}$ when $u_2 v_2 \in E_2$.
Its adjacency matrix is $A_1 \otimes I_2 + I_1 \otimes A_2$.

For abelian Cayley graphs, this arises from the direct product:
if $G = G_1 \times G_2$ and $S = (S_1 \times \{0\}) \cup (\{0\} \times S_2)$,
then $\Cay(G,S) = \Cay(G_1,S_1) \,\square\, \Cay(G_2,S_2)$.

\begin{proposition}[Eigenvalues of Cartesian products]
\label{prop:cartesian-eig}
With notation as above, the eigenvalues of $A_{G_1 \times G_2}
= A_1 \otimes I_2 + I_1 \otimes A_2$ are
\[
  \lambda_{(\chi_1,\chi_2)} = \lambda_{\chi_1}^{(1)} + \lambda_{\chi_2}^{(2)},
  \qquad \chi_i \in \Ghat_i,
\]
and these all lie in $\Q(\zeta_N)^+$ where $N = |G_1||G_2|$.
\end{proposition}

\begin{proof}
The characters of $G_1 \times G_2$ are $\chi_{(\chi_1,\chi_2)}(g_1,g_2)
= \chi_1(g_1)\chi_2(g_2)$.  Evaluating the character sum over
$S = (S_1 \times \{0\}) \cup (\{0\} \times S_2)$ gives
$\lambda_{(\chi_1,\chi_2)} = \sum_{s_1 \in S_1} \chi_1(s_1)
+ \sum_{s_2 \in S_2} \chi_2(s_2)
= \lambda_{\chi_1}^{(1)} + \lambda_{\chi_2}^{(2)}$.
Since $\lambda_{\chi_i}^{(i)} \in \Q(\zeta_{N_i})^+ \subseteq \Q(\zeta_N)^+$,
the sum lies in $\Q(\zeta_N)^+$.
\end{proof}

\begin{proposition}[Tensor product comparison]
\label{prop:tensor}
Let $G = G_1 \times G_2$ and $S = S_1 \oplus S_2$ as above.  There is an
injective $K$-algebra homomorphism
\[
\iota : \calA_K(\Cay(G,S)) \hookrightarrow
\calA_K(\Cay(G_1,S_1)) \otimes_K \calA_K(\Cay(G_2,S_2)),
\]
where $G = G_1\times G_2$ and $S = S_1\oplus S_2$.
In particular, writing $d_i = \dim_K \calA_K(\Cay(G_i,S_i))$,
\[
  \dim_K \calA_K(\Cay(G,S)) \;\le\; d_1 \cdot d_2,
\]
with equality if and only if all pairwise sums $\lambda^{(1)} +
\lambda^{(2)}$ of distinct eigenvalues of $A_1$ and $A_2$
are themselves distinct (as elements of $\overline{K}$).
\end{proposition}

\begin{proof}
The algebras $\calA_K(\Cay(G_1,S_1)) = K[A_1 \otimes I_2]$ and
$\calA_K(\Cay(G_2,S_2)) = K[I_1 \otimes A_2]$ are commuting subalgebras
of $M_{N_1 N_2}(K)$.  Their product $f(A_1 \otimes I_2) \cdot g(I_1
\otimes A_2) = f(A_1) \otimes g(A_2)$ generates a copy of
$\calA_K(G_1,S_1) \otimes_K \calA_K(G_2,S_2)$ inside $M_{N_1 N_2}(K)$.
The element $A = A_1 \otimes I_2 + I_1 \otimes A_2$ lies in this tensor
product algebra, so $\calA_K(\Cay(G,S)) = K[A]$ embeds in it via $\iota$.
The dimension bound reflects that the distinct eigenvalues
$\lambda_{\chi_1} + \lambda_{\chi_2}$ are a subset of the
$\dim_1 \cdot \dim_2$ pairwise sums, with strict inequality when sums
coincide.
\end{proof}

\begin{example}[Inequality in Proposition~\ref{prop:tensor}]
$\calA_\Q(C_4) \cong \Q^3$ (dimension $3$) and $\calA_\Q(C_6) \cong \Q^4$
(dimension $4$).  The tensor product has dimension $12$, yet the Cartesian
product $C_4 \,\square\, C_6$ has only $9$ distinct eigenvalues
$\{4,3,2,1,0,-1,-2,-3,-4\} \subset \Q$, so
$\calA_\Q(C_4 \,\square\, C_6) \cong \Q^9$ with dimension $9 < 12$.
\end{example}

\section{Elementary Abelian Groups}
\label{sec:elem}

An \emph{elementary abelian $p$-group} is a group isomorphic to
$(\Z/p)^k$ for a prime $p$.  Its character group is isomorphic to
$(\Z/p)^k$ via $\chi_a(b) = \omega_p^{a \cdot b}$ where $\omega_p$ is a
primitive $p$-th root of unity and $a \cdot b = \sum_i a_i b_i$.

For the \emph{standard} connection set $S = \{e_1, -e_1, \ldots, e_k, -e_k\}$
(standard basis vectors and their negatives), the Cayley graph
$\Cay((\Z/p)^k, S)$ is the $k$-dimensional Hamming graph over $\Z/p$.
The eigenvalue at character $\chi_a$ is
\begin{equation}
\label{eq:elem-eig}
  \lambda_a = \sum_{i=1}^k (\omega_p^{a_i} + \omega_p^{-a_i})
  = \sum_{i=1}^k 2\cos(2\pi a_i/p),
\end{equation}
so $\lambda_a = \lambda_{a'}$ whenever $\{|a_1|_p, \ldots, |a_k|_p\}
= \{|a_1'|_p, \ldots, |a_k'|_p\}$ as multisets, where $|a_i|_p =
\min(a_i, p - a_i)$.

\begin{theorem}[Elementary abelian groups: spectrum classification]
\label{thm:elem}
Let $G = (\Z/p)^k$ and $S$ the standard connection set above.
\begin{enumerate}[label=\textup{(\roman*)}]
  \item For $p = 2$: $\omega_2 = -1$, all eigenvalues $\lambda_a =
        k - 2\wt(a) \in \Z$, and
        $\calA_\Q(\Cay(G,S)) \cong \Q^{k+1}$.

  \item For $p = 3$: $\omega_3 + \omega_3^{-1} = -1$, so each term in
        ~\eqref{eq:elem-eig} is $-1$ (for $a_i \ne 0$) or $2$ (for
        $a_i = 0$).  All eigenvalues are integers in $\{-k, -k+3, \ldots, 2k\}$
        and $\calA_\Q(\Cay(G,S)) \cong \Q^{k+1}$.

  \item For $p = 5$: $\omega_5 + \omega_5^{-1} = (\sqrt{5}-1)/2$, so
        each term $2\cos(2\pi a_i/5)$ lies in $\Q(\sqrt{5})$.  Eigenvalue
        sums $\sum_i 2\cos(2\pi a_i/5)$ therefore lie in $\Q(\sqrt{5})$,
        and the Galois orbits pair conjugate sums.  Singleton orbits are
        those sums invariant under $\sqrt{5} \mapsto -\sqrt{5}$.  The
        quadratic minimal polynomials depend on the specific eigenvalue
        (see Example~\ref{ex:elem-p5}); they are not all of a single form.

  \item For $p \ge 7$: eigenvalues lie in $\Q(\zeta_p)^+$, which has
        degree $(p-1)/2 \ge 3$ over $\Q$; the spectral algebra generally
        contains factors of degree $> 2$.
\end{enumerate}
\end{theorem}

\begin{proof}
(i) is the Hamming cube (Theorem~\ref{thm:hamming} below).  For~(ii):
$2\cos(2\pi/3) = -1$, so each term in~\eqref{eq:elem-eig} equals $-1$
if $a_i \ne 0$ and $2$ if $a_i = 0$.  The eigenvalue depends only on
$|\{i : a_i \ne 0\}| = \wt(a)$, giving $\lambda_a = 2(k-\wt(a)) - \wt(a)
= 2k - 3\wt(a) \in \Z$.  The $k+1$ distinct values are rational, so all
Galois orbits are singletons and the spectral algebra is $\Q^{k+1}$.
For~(iii): $2\cos(2\pi/5) = (\sqrt{5}-1)/2 \in \Q(\sqrt{5})$.
The distinct eigenvalues are $\Q(\sqrt{5})$-valued, and the Galois group
$\Gal(\Q(\sqrt{5})/\Q) = \{1, \sigma\}$ acts by $\sigma(\sqrt{5}) = -\sqrt{5}$.
The orbits are pairs $\{\lambda_a, \sigma(\lambda_a)\}$ unless
$\lambda_a \in \Q$.  For~(iv): $[\Q(\zeta_p)^+:\Q] = (p-1)/2 \ge 3$,
and the individual summands $2\cos(2\pi a_i/p)$ are conjugate under the
degree-$(p-1)/2$ extension, so the eigenvalue sums generally require
extensions of high degree.
\end{proof}

\begin{example}[$p = 5$, $k = 2$]
\label{ex:elem-p5}
$G = (\Z/5)^2$, $S = \{e_1, -e_1, e_2, -e_2\}$.  The eigenvalues
$\lambda_{(a_1,a_2)} = 2\cos(2\pi a_1/5) + 2\cos(2\pi a_2/5)$ for
$a_i \in \{0,1,2\}$ (since $|\cdot|_5$ identifies $1 \leftrightarrow 4$,
$2 \leftrightarrow 3$).  Setting $\alpha = 2\cos(2\pi/5) = (\sqrt{5}-1)/2$
and $\beta = 2\cos(4\pi/5) = -(\sqrt{5}+1)/2$:
\begin{center}
\renewcommand{\arraystretch}{1.15}
\begin{tabular}{cc|c}
\toprule
$|a_1|_5$ & $|a_2|_5$ & $\lambda$ \\
\midrule
$0$ & $0$ & $4$ \\
$0$ & $1$ & $2+\alpha$ \\
$0$ & $2$ & $2+\beta$ \\
$1$ & $1$ & $2\alpha$ \\
$1$ & $2$ & $\alpha+\beta = -1$ \\
$2$ & $2$ & $2\beta$ \\
\bottomrule
\end{tabular}
\end{center}
Since $\sigma(\alpha) = \beta$ and $\sigma(\beta) = \alpha$:
the rational eigenvalues $\{4, -1\}$ form singleton orbits, and
$\{2+\alpha, 2+\beta\}$, $\{2\alpha, 2\beta\}$ are conjugate pairs.
Hence $\calA_\Q(\Cay((\Z/5)^2,S)) \cong \Q^2 \times \Q(\sqrt{5})^2$,
with dimension $6$ and $2^4 = 16$ idempotents.
\end{example}

\section{Bose--Mesner Algebras and the Cayley Scheme}
\label{sec:bose}

An \emph{association scheme} on a finite set $X$ is a partition of
$X \times X$ into classes $R_0, R_1, \ldots, R_d$ satisfying the standard
axioms: $R_0$ is the diagonal, each $R_i$ is symmetric, and the structure
constants $p_{ij}^k = |R_k \cap (R_i \times R_j)|$ are well-defined.
The \emph{Bose--Mesner algebra} $\mathcal{M}$ is the commutative
$K$-algebra spanned by the adjacency matrices $A_0, \ldots, A_d$, with
$A_k$ being the $(0,1)$-matrix of relation $R_k$.

For a transitive action of a group on itself (by translation), any Cayley
graph $\Cay(G,S)$ on a finite abelian group $G$ arises from a
\emph{Cayley association scheme}: the classes are $R_s = \{(g,g+s) : g
\in G\}$ for $s \in G$, and the Bose--Mesner algebra $\mathcal{M}$ is
the full commutative algebra $K[P_g : g \in G]$ generated by all
translation matrices $P_g$ (where $(P_g)_{h,h'} = \mathbf{1}[h' = h+g]$).

\begin{proposition}[Spectral algebra as a Bose--Mesner subalgebra]
\label{prop:bose}
The spectral algebra $\calA_K(\Cay(G,S))$ is a subalgebra of the
Bose--Mesner algebra $\mathcal{M}$ of the Cayley association scheme on $G$.
Specifically, $A = \sum_{s \in S} P_s \in \mathcal{M}$, so
$\calA_K(\Cay(G,S)) = K[A] \subseteq \mathcal{M}$.

In general $\calA_K(\Cay(G,S)) \subsetneq \mathcal{M}$: the complete Cayley
graph $\Cay(G, G\setminus\{0\})$ has only two distinct eigenvalues $N-1$
and $-1$, giving $\calA_K(\Cay(G,G\setminus\{0\})) \cong K^2$, which has
dimension $2 \ll N = \dim_K \mathcal{M}$.  To generate all of $\mathcal{M}$
one needs all $N$ translation matrices $\{P_g : g \in G\}$, not just the
single adjacency-matrix sum $\sum_{s \in S} P_s$.
\end{proposition}

\begin{proof}
Each translation $P_s$ lies in $\mathcal{M}$ by the association scheme
axioms.  Their sum $A = \sum_{s \in S} P_s$ therefore lies in $\mathcal{M}$,
and $K[A] \subseteq \mathcal{M}$.  The complete Cayley graph has eigenvalues
$N-1$ (for $\chi_0$) and $-1$ (for all $\chi \ne \chi_0$), with minimal
polynomial $(x-(N-1))(x+1)$; the resulting spectral algebra has dimension
$2$, confirming $K[A] \ne \mathcal{M}$ for $N \ge 3$.
\end{proof}

\begin{remark}
The Bose--Mesner algebra of the Cayley scheme on $G$ is isomorphic to
the group algebra $K[G]$ via $P_g \leftrightarrow g$.  The character-orbit
decomposition of Theorem~\ref{thm:main} is therefore a specialisation of
the classical character-sum Wedderburn decomposition of group algebras:
$K[G] \cong \prod_{\calO \in \calO_K(G)} L_\calO$ where $L_\calO$ is the
field factor corresponding to the Galois orbit $\calO$ of characters.  The
spectral algebra $\calA_K(\Cay(G,S))$ is the image of $K[G]$ under the
map $g \mapsto A^{f(g)}$ where $f$ encodes the connection set.
\end{remark}

\section{Further Examples}
\label{sec:examples}

\subsection{The Hamming cube}

\begin{sloppypar}
The $k$-dimensional Hamming cube
$Q_k = \Cay((\Z/2)^k, \{e_1, \ldots, e_k\})$ has
characters $\chi_a(b) = (-1)^{a \cdot b}$ giving
eigenvalues $\lambda_a = k - 2\wt(a)$.
\end{sloppypar}

\begin{theorem}[Hamming cube spectral algebra]
\label{thm:hamming}
$\calA_\Q(Q_k) \cong \Q^{k+1}$ for every $k \ge 1$.
\end{theorem}

\begin{proof}
The $k+1$ values $k, k-2, k-4, \ldots, -k$ are distinct integers.
Since all eigenvalues are rational, each orbit is a singleton and
$\Psi_\calO = x - (k - 2j)$ for $j = 0, 1, \ldots, k$.
Theorem~\ref{thm:main} gives
$\calA_\Q(Q_k) \cong \prod_{j=0}^k \Q[x]/(x-(k-2j)) \cong \Q^{k+1}$.
\end{proof}

Figure~\ref{fig:Qk-eigenvalues} shows the eigenvalue positions and
multiplicities for $Q_k$, $k = 1, \ldots, 5$.

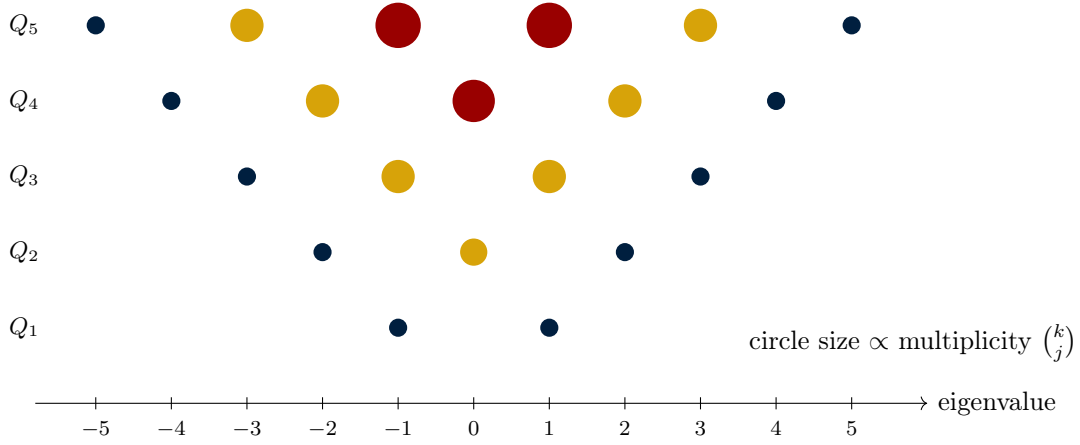
\begin{figure*}[ht]
\centering
\begin{tikzpicture}[scale=1.0]
\draw[->,thin] (-5.8,0) -- (6.0,0) node[right,font=\footnotesize] {eigenvalue};
\foreach \x in {-5,-4,-3,-2,-1,0,1,2,3,4,5}{
  \draw[thin] (\x,-0.08) -- (\x,0.08);
  \node[below,font=\tiny] at (\x,-0.10) {$\x$};
}
\pgfmathsetmacro\ky{1.0}
\fill[mynavy] (1,\ky) circle (0.12);
\fill[mynavy] (-1,\ky) circle (0.12);
\node[left,font=\scriptsize] at (-5.6,\ky) {$Q_1$};

\pgfmathsetmacro\ky{2.0}
\fill[mynavy] (2,\ky) circle (0.12);
\fill[mygold] (0,\ky) circle (0.18);
\fill[mynavy] (-2,\ky) circle (0.12);
\node[left,font=\scriptsize] at (-5.6,\ky) {$Q_2$};

\pgfmathsetmacro\ky{3.0}
\fill[mynavy] (3,\ky) circle (0.12);
\fill[mygold] (1,\ky) circle (0.22);
\fill[mygold] (-1,\ky) circle (0.22);
\fill[mynavy] (-3,\ky) circle (0.12);
\node[left,font=\scriptsize] at (-5.6,\ky) {$Q_3$};

\pgfmathsetmacro\ky{4.0}
\fill[mynavy] (4,\ky) circle (0.12);
\fill[mygold] (2,\ky) circle (0.22);
\fill[mycrimson] (0,\ky) circle (0.28);
\fill[mygold] (-2,\ky) circle (0.22);
\fill[mynavy] (-4,\ky) circle (0.12);
\node[left,font=\scriptsize] at (-5.6,\ky) {$Q_4$};

\pgfmathsetmacro\ky{5.0}
\fill[mynavy] (5,\ky) circle (0.12);
\fill[mygold] (3,\ky) circle (0.22);
\fill[mycrimson] (1,\ky) circle (0.30);
\fill[mycrimson] (-1,\ky) circle (0.30);
\fill[mygold] (-3,\ky) circle (0.22);
\fill[mynavy] (-5,\ky) circle (0.12);
\node[left,font=\scriptsize] at (-5.6,\ky) {$Q_5$};

\node[right,font=\footnotesize] at (3.5,0.80) {circle size $\propto$ multiplicity $\binom{k}{j}$};
\end{tikzpicture}
\caption{Eigenvalue positions of $Q_k$ for $k = 1, \ldots, 5$ on the
  integer number line.  The $k+1$ distinct eigenvalues $k, k-2, \ldots, -k$
  (all rational integers) are shown as dots; the radius is proportional to
  the multiplicity $\binom{k}{j}$ of eigenvalue $k-2j$.  All orbits are
  singletons, giving $\calA_\Q(Q_k) \cong \Q^{k+1}$.  The central
  eigenvalue (at $0$ for even $k$) carries the largest multiplicity.}
\label{fig:Qk-eigenvalues}
\end{figure*}

\subsection{The group $\Z/6 \times \Z/2$}

Let $G = \Z/6 \times \Z/2$ and $S = \{(1,0), (-1,0), (0,1)\}$.
The group has order $N = 12$ and $S$ generates $G$ (since $(1,0)$ generates
$\Z/6$ and $(0,1)$ generates $\Z/2$).  Characters are $\chi_{(a,b)}$ with
$\chi_{(a,b)}(g,h) = \zeta_6^{ag}(-1)^{bh}$.  The eigenvalue is
\[
  \lambda_{(a,b)}
  = \zeta_6^a + \zeta_6^{-a} + (-1)^b
  = 2\cos(\pi a/3) + (-1)^b.
\]

\begin{example}[Orbit analysis for $\Z/6 \times \Z/2$]
\label{ex:Z6Z2}
Computing $\lambda_{(a,b)}$ for $a = 0, \ldots, 5$ and $b = 0, 1$:
\begin{center}
\renewcommand{\arraystretch}{1.15}
\small
\begin{tabular}{cc|c|cc|c}
\toprule
$a$ & $b$ & $\lambda$ & $a$ & $b$ & $\lambda$ \\
\midrule
$0$ & $0$ & $3$  & $3$ & $0$ & $-1$ \\
$0$ & $1$ & $1$  & $3$ & $1$ & $-3$ \\
$1$ & $0$ & $2$  & $4$ & $0$ & $0$ \\
$1$ & $1$ & $0$  & $4$ & $1$ & $-2$ \\
$2$ & $0$ & $0$  & $5$ & $0$ & $2$ \\
$2$ & $1$ & $-2$ & $5$ & $1$ & $0$ \\
\bottomrule
\end{tabular}
\end{center}
The seven distinct values $\{-3, -2, -1, 0, 1, 2, 3\} \subset \Z$ are all
rational.  Every Galois orbit is a singleton, and
$\calA_\Q(\Cay(\Z/6 \times \Z/2, S)) \cong \Q^7$.
The dimension is $7$; the idempotent count is $2^7 = 128$.
Multiplicity of $0$: it occurs at $(a,b) \in \{(1,1),(2,0),(4,0),(5,1)\}$,
so eigenvalue $0$ has algebraic multiplicity $4$.
\end{example}

\subsection{The group $\Z/5 \times \Z/2$}

Let $G = \Z/5 \times \Z/2$ and $S = \{(1,0), (-1,0), (0,1)\}$.
Set $\alpha = 2\cos(2\pi/5) = (\sqrt{5}-1)/2$ and $\beta = 2\cos(4\pi/5)
= -(\sqrt{5}+1)/2$ (so $\alpha + \beta = -1$ and $\alpha\beta = -1$).
The eigenvalue $\lambda_{(a,b)} = 2\cos(2\pi a/5) + (-1)^b$.

\begin{example}[Orbit analysis for $\Z/5 \times \Z/2$]
\label{ex:Z5Z2}
The six distinct eigenvalues and their Galois orbits:
\begin{center}
\renewcommand{\arraystretch}{1.10}
\footnotesize
\begin{tabular}{cc|cc|c}
\toprule
Eigenvalue & Value & Conjugate & Value & $\Psi$ \\
\midrule
$\lambda_{(0,0)}$ & $3$ & (rational, singleton) & & $x-3$ \\
$\lambda_{(0,1)}$ & $1$ & (rational, singleton) & & $x-1$ \\
$\lambda_{(1,0)}$ & $\alpha+1$ & $\lambda_{(2,0)} = \beta+1$ & $(1-\sqrt{5})/2$ & $x^2-x-1$ \\
$\lambda_{(1,1)}$ & $\alpha-1$ & $\lambda_{(2,1)} = \beta-1$ & $-(3+\sqrt{5})/2$ & $x^2+3x+1$ \\
\bottomrule
\end{tabular}
\end{center}
(Note: $\lambda_{(3,b)} = \lambda_{(2,b)}$ and $\lambda_{(4,b)} = \lambda_{(1,b)}$
by the symmetry $2\cos(6\pi/5) = 2\cos(4\pi/5)$.)  Thus
\[
  \calA_\Q\!\bigl(\Cay(\Z/5 \times \Z/2, S)\bigr)
  \;\cong\; \Q^2 \times \Q(\sqrt{5})^2,
\]
with dimension $6$ and $2^4 = 16$ idempotents.  Both $\Q(\sqrt{5})$ factors
are isomorphic as abstract fields, yet they are distinct Wedderburn summands
corresponding to different Galois orbits of eigenvalues.
\end{example}

Figure~\ref{fig:orbit-grid} shows the orbit structure for the
$\Z/5 \times \Z/2$ example as a coloured grid.

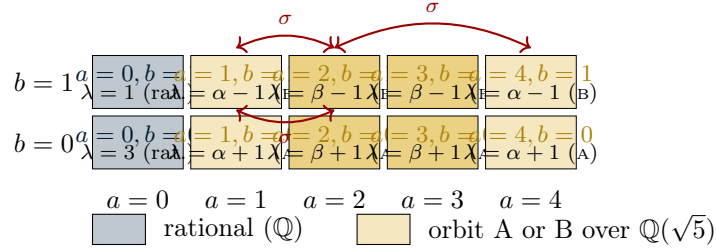
\begin{figure*}[ht]
\centering
\begin{tikzpicture}[scale=1.0]
\newcommand{\drawgridcell}[5]{%
  \fill[#3!25] (#1,#2) rectangle (#1+1.20,#2+0.70);
  \draw[black,very thin] (#1,#2) rectangle (#1+1.20,#2+0.70);
}
\fill[mynavy!25] (0.00,0.80) rectangle (1.20,1.50);
\draw[black,very thin] (0.00,0.80) rectangle (1.20,1.50);
\node[font=\scriptsize,mynavy!80!black] at (0.60,1.24) {$a=0,b=1$};
\node[font=\tiny] at (0.60,1.00) {$\lambda=1$ (rat.)};

\fill[mynavy!25] (0.00,0.00) rectangle (1.20,0.70);
\draw[black,very thin] (0.00,0.00) rectangle (1.20,0.70);
\node[font=\scriptsize,mynavy!80!black] at (0.60,0.44) {$a=0,b=0$};
\node[font=\tiny] at (0.60,0.20) {$\lambda=3$ (rat.)};

\fill[mygold!25] (1.30,0.80) rectangle (2.50,1.50);
\draw[black,very thin] (1.30,0.80) rectangle (2.50,1.50);
\node[font=\scriptsize,mygold!80!black] at (1.90,1.24) {$a=1,b=1$};
\node[font=\tiny] at (1.90,1.00) {$\lambda=\alpha-1$ (\textsc{b})};

\fill[mygold!25] (1.30,0.00) rectangle (2.50,0.70);
\draw[black,very thin] (1.30,0.00) rectangle (2.50,0.70);
\node[font=\scriptsize,mygold!80!black] at (1.90,0.44) {$a=1,b=0$};
\node[font=\tiny] at (1.90,0.20) {$\lambda=\alpha+1$ (\textsc{a})};

\fill[mygold!50] (2.60,0.80) rectangle (3.80,1.50);
\draw[black,very thin] (2.60,0.80) rectangle (3.80,1.50);
\node[font=\scriptsize,mygold!80!black] at (3.20,1.24) {$a=2,b=1$};
\node[font=\tiny] at (3.20,1.00) {$\lambda=\beta-1$ (\textsc{b})};

\fill[mygold!50] (2.60,0.00) rectangle (3.80,0.70);
\draw[black,very thin] (2.60,0.00) rectangle (3.80,0.70);
\node[font=\scriptsize,mygold!80!black] at (3.20,0.44) {$a=2,b=0$};
\node[font=\tiny] at (3.20,0.20) {$\lambda=\beta+1$ (\textsc{a})};

\fill[mygold!50] (3.90,0.80) rectangle (5.10,1.50);
\draw[black,very thin] (3.90,0.80) rectangle (5.10,1.50);
\node[font=\scriptsize,mygold!80!black] at (4.50,1.24) {$a=3,b=1$};
\node[font=\tiny] at (4.50,1.00) {$\lambda=\beta-1$ (\textsc{b})};

\fill[mygold!50] (3.90,0.00) rectangle (5.10,0.70);
\draw[black,very thin] (3.90,0.00) rectangle (5.10,0.70);
\node[font=\scriptsize,mygold!80!black] at (4.50,0.44) {$a=3,b=0$};
\node[font=\tiny] at (4.50,0.20) {$\lambda=\beta+1$ (\textsc{a})};

\fill[mygold!25] (5.20,0.80) rectangle (6.40,1.50);
\draw[black,very thin] (5.20,0.80) rectangle (6.40,1.50);
\node[font=\scriptsize,mygold!80!black] at (5.80,1.24) {$a=4,b=1$};
\node[font=\tiny] at (5.80,1.00) {$\lambda=\alpha-1$ (\textsc{b})};

\fill[mygold!25] (5.20,0.00) rectangle (6.40,0.70);
\draw[black,very thin] (5.20,0.00) rectangle (6.40,0.70);
\node[font=\scriptsize,mygold!80!black] at (5.80,0.44) {$a=4,b=0$};
\node[font=\tiny] at (5.80,0.20) {$\lambda=\alpha+1$ (\textsc{a})};

\node[below,font=\footnotesize] at (0.60,-0.15) {$a=0$};
\node[below,font=\footnotesize] at (1.90,-0.15) {$a=1$};
\node[below,font=\footnotesize] at (3.20,-0.15) {$a=2$};
\node[below,font=\footnotesize] at (4.50,-0.15) {$a=3$};
\node[below,font=\footnotesize] at (5.80,-0.15) {$a=4$};
\node[left,font=\footnotesize] at (-0.10,1.15) {$b=1$};
\node[left,font=\footnotesize] at (-0.10,0.35) {$b=0$};

\draw[mycrimson,thick,<->] (1.90,1.60) to[bend left=25] node[above,font=\tiny,mycrimson] {$\sigma$} (3.20,1.60);
\draw[mycrimson,thick,<->] (1.90,0.78) to[bend right=25] node[below,font=\tiny,mycrimson] {$\sigma$} (3.20,0.78);
\draw[mycrimson,thick,<->] (5.80,1.60) to[bend right=25] node[above,font=\tiny,mycrimson] {$\sigma$} (3.20,1.60);

\fill[mynavy!25] (0,-1.00) rectangle (0.70,-0.60);
\draw[black,thin] (0,-1.00) rectangle (0.70,-0.60);
\node[right,font=\footnotesize] at (0.80,-0.80) {rational ($\Q$)};
\fill[mygold!25] (3.50,-1.00) rectangle (4.20,-0.60);
\draw[black,thin] (3.50,-1.00) rectangle (4.20,-0.60);
\node[right,font=\footnotesize] at (4.30,-0.80) {orbit A or B over $\Q(\sqrt{5})$};
\end{tikzpicture}
\caption{Galois orbit grid for $\Cay(\Z/5 \times \Z/2, S)$ with
  $S = \{(1,0),(-1,0),(0,1)\}$.  Each cell represents one character
  $\chi_{(a,b)}$ with the corresponding eigenvalue $\lambda_{(a,b)}$.
  Navy cells carry rational eigenvalues (singleton orbits); gold cells
  carry quadratic irrational eigenvalues in $\Q(\sqrt{5})$.
  Lighter gold (orbits at $a=1,4$) and darker gold ($a=2,3$) distinguish
  the two conjugate pairs.  The Galois automorphism $\sigma$ (swapping
  $\sqrt{5} \leftrightarrow -\sqrt{5}$) is shown as crimson arcs connecting
  conjugate cells.  The spectral algebra is
  $\calA_\Q \cong \Q^2 \times \Q(\sqrt{5})^2$.}
\label{fig:orbit-grid}
\end{figure*}

\subsection{Grid graphs}

The grid $C_n \,\square\, C_m = \Cay(\Z/n \times \Z/m, S_\square)$ with
$S_\square = \{(\pm 1,0),(0,\pm 1)\}$ has eigenvalues
$\lambda_{(a,b)} = 2\cos(2\pi a/n) + 2\cos(2\pi b/m)$ for $0 \le a < n$,
$0 \le b < m$.

\begin{example}[$C_4 \,\square\, C_6$]
$N = 24$.  The distinct eigenvalues from
$\{2\cos(\pi a/2) + 2\cos(\pi b/3) : 0 \le a \le 2,\, 0 \le b \le 3\}$
are $\{-4, -3, -2, -1, 0, 1, 2, 3, 4\}$, all rational integers.
$\calA_\Q(C_4 \,\square\, C_6) \cong \Q^9$; dim $= 9$, idempotents
$= 2^9 = 512$.
\end{example}

\begin{example}[$C_5 \,\square\, C_5$]
The standard connection set
$S = \{(1,0),(-1,0),(0,1),(0,-1)\}$ gives
$C_5 \,\square\, C_5$.  Eigenvalues are $2\cos(2\pi a/5) + 2\cos(2\pi b/5)$.
\begin{sloppypar}
Since $2\cos(2\pi a/5)$ takes three values over $a\in\Z/5$
(namely $2$, $\alpha = (\sqrt{5}-1)/2$, $\beta = -(\sqrt{5}+1)/2$,
with $a=1,4$ giving $\alpha$ and $a=2,3$ giving $\beta$),
the six distinct pairwise sums $\alpha_a + \alpha_b$:\end{sloppypar}
\[
4,\quad 2+\alpha,\quad 2+\beta,\quad 2\alpha,\quad
\alpha+\beta = -1,\quad 2\beta.
\]
(The values $-2+\alpha$, $-2+\beta$, $-4$ cannot occur since $-2 =
2\cos(\pi)$ requires an element of order $2$ in $\Z/5$, which does not
exist.)
Galois orbits: $\{4\}$, $\{-1\}$ (both rational), $\{2+\alpha, 2+\beta\}$,
$\{2\alpha, 2\beta\}$ (both quadratic conjugate pairs).
$\calA_\Q(C_5 \,\square\, C_5) \cong \Q^2 \times \Q(\sqrt{5})^2$;
dim $= 6$, idempotents $= 2^4 = 16$.
\end{example}

\section{Characteristic $p$}
\label{sec:charp}

\subsection{The semisimple case}

When $\ch(K) = p \nmid N$, the polynomial $x^N - 1$ is still separable,
so Theorem~\ref{thm:main} holds verbatim.  The Galois orbits are determined
by the Frobenius automorphism $x \mapsto x^q$ (for $K = \F_q$) acting on
the eigenvalues in $\overline{\F}_q$.

\begin{proposition}[Finite field case]
\label{prop:Fq}
Let $K = \F_q$, $q = p^f$, $p \nmid N$.  Then $\calA_{\F_q}(\Cay(G,S))$
is semisimple and decomposes as in Theorem~\ref{thm:main}, with orbits
under the Frobenius $\mathrm{Frob}_q : x \mapsto x^q$.  Each Wedderburn
factor is an extension $\F_{q^d}$ of $\F_q$ for some $d \ge 1$.
\end{proposition}

\begin{example}[$C_7$ over $\F_2$]
The minimal polynomial of $A_{C_7}$ over $\Q$ is
$m_{C_7,\Q}(x) = (x-2)(x^3+x^2-2x-1)$.  Reducing mod $2$:
$x(x^3+x^2+1)$.  Since $x^3+x^2+1$ is irreducible over $\F_2$
(no roots in $\{0,1\}$), $\calA_{\F_2}(C_7) \cong \F_2 \times \F_8$.
\end{example}

\begin{example}[$Q_3$ over $\F_3$]
$Q_3$ has eigenvalues $3, 1, -1, -3$.  Modulo $3$: $0, 1, -1, 0$.
The distinct values in $\F_3$ are $\{0, 1, -1\} = \{0, 1, 2\}$, so
the minimal polynomial over $\F_3$ is $x(x-1)(x+1) = x(x^2-1) = x^3-x$.
$\calA_{\F_3}(Q_3) \cong \F_3^3$.  The two distinct eigenvalues $3 \equiv 0$
and $-3 \equiv 0$ in $\F_3$ collapse to the same value, reducing the
dimension from $4$ to $3$.
\end{example}

\subsection{The non-semisimple case}

When $\ch(K) = p \mid N$, write $N = p^a M$ with $p \nmid M$.

\begin{proposition}[Characteristic dividing the group order]
\label{prop:charp}
Let $\ch(K) = p \mid N$, $N = p^a M$, $p \nmid M$.  Let $G'$ be the
prime-to-$p$ quotient of $G$ (of order $M$) and $\pi : G \to G'$ the
quotient map.  Define the weighted connection set $S'$ on $G'$ by
$w(s') = |\{s \in S : \pi(s) = s'\}|$ for $s' \in G'$.  Then
\[
  \calA_K(\Cay(G,S))\big/\Jac(\calA_K(\Cay(G,S)))
  \;\cong\;
  K[x]\big/(m_{G,S,K}^{\mathrm{red}}),
\]
where $m_{G,S,K}^{\mathrm{red}}$ is the squarefree part of $m_{G,S,K}$
and equals the minimal polynomial of the weighted Cayley matrix on $G'$.
\end{proposition}

\begin{proof}
In $\overline{K}[x]$, $x^N - 1 = (x^M - 1)^{p^a}$ by the Frobenius
endomorphism, so every $N$-th root of unity in $\overline{K}$ is an $M$-th
root of unity.  Each character $\chi \in \Ghat$ factors as $\chi = \chi'
\circ \pi$ for a unique $\chi' \in \widehat{G'}$, and the eigenvalue
$\lambda_\chi = \sum_{s \in S} \chi(s) = \sum_{s' \in G'} w(s') \chi'(s')$
equals the eigenvalue of the weighted matrix on $G'$.  Hence
$m_{G,S,K}^{\mathrm{red}} = m_{G',S',K}$, and the Jacobson radical
identity follows from Proposition~\ref{prop:ss}.
\end{proof}

\section{Coding Theory Applications}
\label{sec:coding}

\subsection{Cyclic codes: the cycle-graph case}

\begin{remark}[Scope of the coding-theory connection]
\label{rem:coding-scope}
A cyclic code of length $N$ over $\F_q$ is an ideal in $\F_q[x]/(x^N-1)$.
The spectral algebra $\calA_{\F_q}(\Cay(G,S))$ is \emph{not} in general a
quotient of $\F_q[x]/(x^N-1)$: the minimal polynomial $m_{G,S,\F_q}$ of
$A$ need not divide $x^N-1$.  For the cycle graph $\Cay(\Z/N,\{1,-1\})$
the eigenvalues $2\cos(2\pi k/N)$ are not $N$-th roots of unity, so
$\Psi_{\calO,\F_q}$ is not in general a factor of $x^N-1$.

The connection to cyclic codes is indirect and works cleanly only in the
special case where the eigenvalues of $A$ happen to lie among the $N$-th
roots of unity in $\overline{\F}_q$.  For the cycle graph over $\F_q$ with
$\ch(\F_q) \nmid N$, the eigenvalues $\lambda_k = \zeta_N^k + \zeta_N^{-k}$
are character sums of $N$-th roots of unity, and the factorisation of
$x^N-1$ over $\F_q$ governs the cyclic-code structure.
The BCH bound for cyclic codes over finite fields is standard~\cite{HP03}.
\end{remark}

\begin{example}[$C_7$ over $\F_2$ and the Hamming code]
The cycle $C_7$ over $\F_2$: the minimal polynomial $m_{C_7,\F_2}(x)
= x(x^3+x^2+1)$.  The polynomial $x^3+x^2+1$ is a factor of $x^7-1$
over $\F_2$ ($x^7-1 = (x+1)(x^3+x^2+1)(x^3+x+1)$ over $\F_2$), so the
spectral algebra embeds in $\F_2[x]/(x^7-1)$.  This embedding connects the
orbit $\calO$ with minimal polynomial $x^3+x^2+1$ to a subspace of the
cyclic code ring.  Note: the standard $[7,4,3]$ Hamming code has generator
$x^3+x+1$ (the other cubic factor), not $x^3+x^2+1$.  The relationship
between the spectral orbit $\calO_7$ and the Hamming code is an analogy
suggested by this common-factor structure; a precise duality-map argument
is outside the scope of this paper.
\end{example}

\section{The Companion Paper Connection}
\label{sec:companion}

The companion paper~\cite{MKB26} defines, for squarefree odd $n$, the
homomorphism quotient $\varrho(n) = \varphi(n)/2^{\omega(n)}$ and shows it
equals $\prod_{p \mid n}(p-1)/2$.  In the spectral algebra language
of the cyclic case (Section~\ref{sec:cyclic}), this is the product of
field degrees of the prime-indexed Wedderburn summands.

For the cyclic case $G = \Z/n$ and $S = \{1,-1\}$, the companion
paper~\cite{MKB26} gives $\varrho(n) = \varphi(n)/2^{\omega(n)}
= \prod_{p \mid n}(p-1)/2$, which is the product of field degrees of the
prime-indexed Wedderburn summands of $\calA_\Q(C_n)$ (Section~\ref{sec:cyclic}).

\begin{remark}[Why no product formula for general $G$]
One might hope to extend $\varrho$ to general abelian $G$ by setting
$\varrho(G,S) = \prod_\calO [L_\calO:\Q]$.  This product does not equal
$\dim_\Q \calA_\Q(\Cay(G,S)) = \sum_\calO [L_\calO:\Q]$ in general.
For example, $\calA_\Q(C_5) \cong \Q \times \Q(\sqrt{5})$ has
$\sum [L_\calO:\Q] = 1+2 = 3$ while $\prod [L_\calO:\Q] = 1 \cdot 2 = 2$.
A natural generalisation of $\varrho$ from the companion paper to
non-cyclic groups does not follow directly from the Wedderburn decomposition
and is left as an open problem.
\end{remark}

\section{Galois Module Structure of Summands}
\label{sec:galois-structure}

For each orbit $\calO \in \calO_\Q(\Cay(G,S))$, the Wedderburn summand
$L_\calO = \Q(\lambda_\calO)$ is an abelian number field, and its Galois
group over $\Q$ is computable from the stabiliser of $\lambda_\calO$ in
$(\Z/N\Z)^*/\{\pm 1\}$.

\begin{proposition}[Galois group of a summand]
\label{prop:galois-group}
For $K = \Q$, the extension $L_\calO/\Q$ is Galois with group
\[
  \Gal(L_\calO/\Q) \;\cong\;
  (\Z/N\Z)^* \Big/ \mathrm{Stab}_{(\Z/N\Z)^*}(\lambda_\calO),
\]
where $\mathrm{Stab}(\lambda_\calO) = \{a \in (\Z/N\Z)^* :
\lambda_{\chi^a} = \lambda_\chi\}$.  Since $\lambda_\chi = \lambda_{\chi^{-1}}$
for all $\chi$ (the eigenvalues are real), $-1 \in \mathrm{Stab}(\lambda_\calO)$
always holds, so $\langle -1 \rangle \subseteq \mathrm{Stab}(\lambda_\calO)$,
and equivalently
\[
  \Gal(L_\calO/\Q) \;\cong\;
  \bigl[(\Z/N\Z)^*/\langle -1 \rangle\bigr]
  \Big/ \bigl[\mathrm{Stab}(\lambda_\calO)/\langle -1\rangle\bigr].
\]
\end{proposition}

\begin{proof}
The group $(\Z/N\Z)^*$ acts on the set $\{\lambda_\chi : \chi \in \Ghat\}$
by $\sigma_a(\lambda_\chi) = \lambda_{\chi^a}$ (Lemma~\ref{lem:galois-action}).
The stabiliser of $\lambda_\calO$ in $(\Z/N\Z)^*$ has index $|\calO|$ by the
orbit-stabiliser theorem.  Since $\lambda_\calO$ generates $L_\calO$ over
$\Q$, the residual action of $(\Z/N\Z)^*/\mathrm{Stab}(\lambda_\calO)$ on
$\calO$ gives a faithful action, hence the isomorphism with $\Gal(L_\calO/\Q)$.
The inclusion $\langle -1 \rangle \subseteq \mathrm{Stab}(\lambda_\calO)$
follows from $\lambda_\chi = \lambda_{\chi^{-1}}$.
\end{proof}

\begin{example}[Galois groups for $C_n$]
For the orbit $\calO_d$ of $C_n$ (indexed by $d \mid n$):
\[
\Gal(\Q(\zeta_d)^+/\Q) \cong (\Z/d\Z)^*/\langle -1 \rangle.
\]
For $d$ prime: this group is cyclic of order $(d-1)/2$.
\begin{sloppypar}
For $d = 12$: $(\Z/12\Z)^* = \{1,5,7,11\} \cong \Z/2 \times \Z/2$,
so $\langle -1 \rangle = \langle 11 \rangle = \{1,11\}$
(since $-1 \equiv 11 \pmod{12}$), and the quotient
$(\Z/12\Z)^*/\{1,11\} \cong \{1,5,7,11\}/\{1,11\} \cong \{1,5\} \cong \Z/2$.
This matches $[\Q(\sqrt{3}):\Q] = 2$.
\end{sloppypar}
\end{example}

\begin{corollary}[Abelianness of Galois groups]
For any abelian Cayley graph $\Cay(G,S)$ over $\Q$, every Wedderburn
summand $L_\calO$ is an \emph{abelian} extension of $\Q$, and
real abelian extension).  By the Kronecker--Weber theorem, every $L_\calO$
is contained in a cyclotomic field; Corollary~\ref{cor:cyclo} gives the
explicit cyclotomic field $\Q(\zeta_N)^+ \supseteq L_\calO$.
\end{corollary}

\section{Spectral Algebras as Graph Isomorphism Invariants}
\label{sec:invariant}

The Wedderburn type of $\calA_K(\Cay(G,S))$ is an isomorphism invariant.

\begin{definition}[Spectral algebra type]
The \emph{type} of $\calA_K(\Cay(G,S))$ is the multiset of $K$-isomorphism
classes of its Wedderburn factors:
$\mathrm{type}_K(\Cay(G,S)) = \{[L_\calO] : \calO \in \calO_K(G,S)\}$.
\end{definition}

\begin{proposition}[Isomorphism invariance]
If $\Cay(G,S) \cong \Cay(G',S')$ as graphs, then
$\mathrm{type}_K(\Cay(G,S)) = \mathrm{type}_K(\Cay(G',S'))$.
\end{proposition}

\begin{proof}
A graph isomorphism conjugates the adjacency matrices, preserving the
minimal polynomial and hence the multiset of orbit minimal polynomials.
\end{proof}

\begin{remark}[Limitation of the invariant]
The spectral algebra type does not distinguish all non-isomorphic Cayley
graphs.  For example, $C_6$ and $K_{3,3}$ (the complete bipartite graph
on $3+3$ vertices) have the same spectrum $\{-2,-1,-1,1,1,2\}$ and hence
the same spectral algebra type $\Q^4$ (four distinct rational eigenvalues).
The spectral algebra sees only the distinct eigenvalue set, not multiplicities.
\end{remark}

\begin{proposition}[When the type determines $S$]
Suppose $G = \Z/N$ is cyclic and $S = \{s, -s\}$ is a two-element
connection set.  Then the type $\mathrm{type}_\Q(\Cay(\Z/N,\{s,-s\}))$
uniquely determines $\gcd(s, N)$, hence the isomorphism class of
$\Cay(\Z/N, \{s,-s\})$.
\end{proposition}

\begin{proof}
Let $d = \gcd(s,N)$ and $N' = N/d$.  The adjacency matrix is
block-diagonal with $d$ identical blocks $A_{C_{N'}}$, so
$K[A] \cong K[A_{C_{N'}}]$ (the algebra generated by identical blocks
equals that of a single block).  The Wedderburn type of $K[A]$ therefore
equals $\mathrm{type}_\Q(C_{N'}) = \prod_{e \mid N'} \Q(\zeta_e)^+$,
with summand field degrees $\{\varphi(e)/2 : e \mid N', e \ge 3\}
\cup \{1\}$.  The value $N'$ is determined from this type as follows:
$\dim_\Q \calA_\Q(\Cay(\Z/N,\{s,-s\})) = \lfloor N'/2\rfloor + 1$,
so $N' = 2(\dim - 1)$ (for odd $N'$) or $N' = 2\dim - 2$ (for even $N'$),
both uniquely determined by the dimension.  Hence $d = N/N' = \gcd(s,N)$.
\end{proof}

\section{The Group Algebra and the Spectral Algebra}
\label{sec:groupalg}

The connection between the spectral algebra and the group algebra $K[G]$
illuminates both structures.

\begin{proposition}[Group algebra Wedderburn decomposition]
\label{prop:groupalg}
For a finite abelian group $G$ of order $N$ and a field $K$ with
$\ch(K) \nmid N$, the group algebra $K[G]$ is semisimple and decomposes as
\[
  K[G] \;\cong\; \prod_{\calO \in \Ghat/\Gal(\overline{K}/K)} L_\calO,
\]
where $\calO$ runs over Galois orbits of characters and
$L_\calO = K(\chi(g) : g \in G,\, \chi \in \calO)$ is the field generated
by the character values of orbit $\calO$.  For $K = \Q$ and $G = \Z/N$
cyclic, the orbit fields are subfields of $\Q(\zeta_N)$.  For non-cyclic
$G$, the orbit fields are also subfields of $\Q(\zeta_N)$, but their
structure depends on the specific character orbits, not on a single generator.
\end{proposition}

\begin{proof}
Since $\ch(K) \nmid N$, the group algebra $K[G]$ is semisimple by
Maschke's theorem~\cite{Lang02}.  For $G \cong \Z/n_1 \times \cdots
\times \Z/n_r$, $K[G] \cong K[x_1,\ldots,x_r]/(x_1^{n_1}-1,\ldots,
x_r^{n_r}-1)$, and each factor is separable.  The Chinese Remainder
Theorem gives the orbit decomposition with one field factor per Galois
orbit of character values.
\end{proof}

\begin{remark}[Spectral algebra as image of group algebra]
The spectral algebra $\calA_K(\Cay(G,S))$ is the image
of $K[G]$ under the ring homomorphism
\[
  \phi_S : K[G] \to M_N(K), \quad g \mapsto P_g,
\]
where $P_g$ is the permutation matrix for translation by $g$.  The
adjacency matrix is $A = \sum_{s \in S} P_s = \phi_S\!\bigl(\sum_{s \in S} g\bigr)$,
so $\calA_K(\Cay(G,S)) = K\!\left[\phi_S\!\bigl(\sum_{s\in S} g\bigr)\right]$
is a quotient of the image of $\phi_S$.  The map $\phi_S$ is an
isomorphism from $K[G]$ to the Bose--Mesner algebra
$\mathcal{M} \cong K[G]$ (as proved in Proposition~\ref{prop:bose}),
and $\calA_K(\Cay(G,S)) \subseteq \mathcal{M}$ is the subalgebra
generated by the single element $\sum_{s \in S} P_s$.
\end{remark}

The key difference between $K[G]$ and $\calA_K(\Cay(G,S))$ is the
number of generators: $K[G]$ needs all $N$ translation matrices, while
$\calA_K(\Cay(G,S))$ uses only their $S$-weighted sum.  The dimension of
$\calA_K(\Cay(G,S))$ is at most $N$ (the dimension of $K[G]$), with
equality only if the minimal polynomial of $A$ has degree $N$, i.e., if
$A$ generates the full Bose--Mesner algebra.

\begin{proposition}[Spectral subalgebras are always proper for symmetric $S$]
\label{prop:full}
For a symmetric connection set $S = -S$, one always has
$\lambda_\chi = \lambda_{\chi^{-1}}$, so the eigenvalue map cannot be
injective on all of $\Ghat$ (since $\chi \ne \chi^{-1}$ in general).
Hence $\calA_K(\Cay(G,S)) \subsetneq \mathcal{M}$ for every non-trivial
symmetric connection set on a group $G$ of order $N \ge 3$.

More precisely, $\dim_K \calA_K(\Cay(G,S))$ equals the number of distinct
eigenvalues, which is at most the number of $\langle -1 \rangle$-orbits on
$\Ghat$ (since $\lambda_\chi = \lambda_{\chi^{-1}}$ forces orbit-level
collapse).  Since $N \ge 3$ implies at least one non-trivial pair
$\{\chi, \chi^{-1}\}$ with $\chi \ne \chi^{-1}$, the number of orbits
is strictly less than $|\Ghat| = N = \dim_K \mathcal{M}$.
\end{proposition}

\begin{proof}
The symmetry $S = -S$ forces $\lambda_\chi = \lambda_{\chi^{-1}}$ for
all $\chi$ (Lemma~\ref{lem:char-sum}(i)), so eigenvalue multiplicity
collapses each pair $\{\chi, \chi^{-1}\}$.  For $N \ge 3$ there exist
$\chi$ with $\chi \ne \chi^{-1}$ (any character of order $> 2$), so the
number of distinct eigenvalues is strictly less than $|\Ghat| = N$, hence
$\dim_K \calA_K < N = \dim_K \mathcal{M}$.
\end{proof}

\begin{remark}[Generating the full Bose--Mesner algebra]
Over fields $K$ with $\ch(K) \nmid N$, the full Bose--Mesner algebra
$\mathcal{M} \cong K[G]$ has dimension $N$ and requires all $N$
translation generators $\{P_g : g \in G\}$.  A single adjacency matrix $A$
can generate at most a $({\lfloor N/2 \rfloor + 1})$-dimensional subalgebra
when $S = -S$.  Asymmetric connection sets (with $S \ne -S$, not considered
here) can potentially generate larger subalgebras.
\end{remark}

\begin{example}[$G = \Z/7$, $S = \{1,-1\}$]
The distinct eigenvalues are $2$, $2\cos(2\pi/7)$,
$2\cos(4\pi/7)$, $2\cos(6\pi/7)$.  Since $4 < 7$,
$\calA_\Q(C_7) \subsetneq \mathcal{M}$ with
$\dim \calA_\Q(C_7) = 4$.
\end{example}

\section{Expansion and the Spectral Gap}
\label{sec:expansion}

A connected $d$-regular graph on $N$ vertices is an \emph{$\varepsilon$-expander}
if the second-largest eigenvalue satisfies $\lambda_1 < d - \varepsilon$.
For Cayley graphs on abelian groups, the spectral algebra encodes the
expansion through its Wedderburn factors.

\begin{definition}[Spectral gap]
For $\Cay(G,S)$ connected and $d$-regular ($d = |S|$), the
\emph{spectral gap} is
\[
  \gamma(\Cay(G,S))
  \coloneqq d - \max\!\bigl\{\lambda_\chi : \chi \ne \chi_0\bigr\}
  = d - \lambda_2,
\]
where $\lambda_2$ is the second-largest eigenvalue.
\end{definition}

The spectral gap determines expansion via $\varepsilon = \gamma$, and
influences random walk mixing times.  In the spectral algebra:

\begin{proposition}[Spectral gap and idempotents]
\label{prop:spectral-gap}
The spectral gap satisfies
\[
  \gamma = d \cdot \min_{\chi \ne \chi_0}
  \left(1 - \frac{\lambda_\chi}{d}\right).
\]
Within the spectral algebra, $\gamma$ measures the distance from the
projection $e_{\chi_0}(A)$ to the next closest eigenvalue projection:
$\gamma = d - \max_{\calO \ne \calO_0} \sup\{x : x \in \calO\}$.
\end{proposition}

\begin{proof}
Immediate from the definition.  The eigenvalue $\lambda_\chi = d$ occurs
only for $\chi = \chi_0$ (principal character), since $|\lambda_\chi| \le d$
with equality iff $\chi(s) = 1$ for all $s \in S$, which forces
$\chi = \chi_0$ when $S$ generates $G$.
\end{proof}

\begin{example}[Spectral gap of $Q_k$]
$Q_k$ is $k$-regular with eigenvalues $k, k-2, \ldots, -k$.  The
second-largest eigenvalue is $k-2$, giving spectral gap $\gamma = 2$.
Since $2/k \to 0$ as $k \to \infty$, the Hamming cubes are
\emph{poor expanders}: the ratio $\gamma/d = 2/k$ tends to zero.
\end{example}

\begin{example}[Spectral gap of $\Cay(\Z/5 \times \Z/2, S)$]
From Example~\ref{ex:Z5Z2}, the connection set $S = \{(1,0),(-1,0),(0,1)\}$
has $|S| = 3$.  The eigenvalues are $3, (\sqrt{5}+1)/2, 1, (\sqrt{5}-3)/2,
(1-\sqrt{5})/2, -(3+\sqrt{5})/2$.  The second largest is
$(\sqrt{5}+1)/2 \approx 1.618$, giving spectral gap
$\gamma = 3 - (\sqrt{5}+1)/2 = (5-\sqrt{5})/2 \approx 1.382$.
\end{example}

\begin{remark}[Ramanujan property]
A $d$-regular graph is \emph{Ramanujan} if $|\lambda_\chi| \le 2\sqrt{d-1}$
for all non-principal characters.  For abelian Cayley graphs, the
eigenvalues $\lambda_\chi \in \Q(\zeta_N)^+$ are bounded by $|S|$;
the Ramanujan condition is automatically satisfied when $|S| = 2$
(the cycle case) since $|\lambda_\chi| \le 2 = 2\sqrt{1}$.  For $|S| \ge 3$,
explicit verification via the Wedderburn factor eigenvalue bounds is required.
\end{remark}

\section{Multiple Cayley Graphs on the Same Group}
\label{sec:multiple}

Different symmetric connection sets $S \ne S'$ on the same group $G$
give different Cayley graphs but with the same ambient Bose--Mesner algebra.
The spectral algebras need not be isomorphic.

\begin{proposition}[Distinct $S$ can give isomorphic spectral algebras]
\label{prop:same-type}
Two Cayley graphs $\Cay(G,S)$ and $\Cay(G,S')$ have
$\calA_K(\Cay(G,S)) \cong \calA_K(\Cay(G,S'))$ as $K$-algebras if and
only if the eigenvalue multisets $\{\lambda_\chi\}_\chi$ and
$\{\lambda'_\chi\}_\chi$ have the same partition into Galois orbits of
the same sizes.
\end{proposition}

\begin{proof}
Both algebras are semisimple products of fields.  The isomorphism type
is determined by the multiset of field extensions, which is determined
by the orbit partition.
\end{proof}

\begin{example}[Two connection sets on $\Z/12$]
\mbox{}
\begin{enumerate}[label=\textup{(\roman*)}]
  \item $\Cay(\Z/12, \{1,-1\}) = C_{12}$ has
    $\calA_\Q \cong \Q^5 \times \Q(\sqrt{3})$, dimension $7$, idempotents
    $2^6 = 64$.
  \item $\Cay(\Z/12, \{3,-3\})$ has eigenvalues
    $2\cos(\pi k/2)$ for $k = 0, \ldots, 5$.
    The distinct values are $\{-2, 0, 2\}$, so
    $\calA_\Q(\Cay(\Z/12, \{3,-3\})) \cong \Q^3$,
    dimension $3$, idempotents $2^3 = 8$.
\end{enumerate}
These algebras are not isomorphic despite $G$ being the same group.
\end{example}

\begin{theorem}[Maximum spectral algebra on $\Z/N$]
\label{thm:max}
For $G = \Z/N$ cyclic, among all symmetric generating sets
$S \subseteq \Z/N \setminus \{0\}$, the maximum dimension of
$\calA_\Q(\Cay(\Z/N, S))$ is $\lfloor N/2 \rfloor + 1$, achieved by the
cycle graph $C_N$ with $S = \{1,-1\}$.
\end{theorem}

\begin{remark}[Non-cyclic groups]
For non-cyclic $G$, the number of $\chi \mapsto \chi^{-1}$ orbits on
$\Ghat$ can exceed $\lfloor N/2 \rfloor + 1$.  For example,
$G = \Z/2 \times \Z/4$ ($N = 8$) has six inversion orbits (four singletons
and two pairs), exceeding $\lfloor 8/2 \rfloor + 1 = 5$.  The dimension
of the spectral algebra for general abelian $G$ is \emph{bounded above}
by the number of $\langle -1 \rangle$-orbits on $\Ghat$; whether this
upper bound is attained by some symmetric $S$ depends on whether a
connection set exists making all inversion-orbit eigenvalues distinct,
which is a separate problem not addressed here.
\end{remark}

\begin{proof}[Proof of Theorem~\ref{thm:max}]
For $G = \Z/N$, the inversion map $\chi_k \mapsto \chi_k^{-1} = \chi_{N-k}$
pairs $k$ with $N-k$.  The number of orbits on $\Ghat$ under this
involution is exactly $\lfloor N/2 \rfloor + 1$ (the set $\{0, 1, \ldots,
\lfloor N/2 \rfloor\}$).  The dimension of $\calA_\Q$ equals the number of
distinct eigenvalues, which is at most $\lfloor N/2 \rfloor + 1$.  For
$S = \{1,-1\}$, the eigenvalues $2\cos(2\pi k/N)$ for $k = 0, \ldots,
\lfloor N/2 \rfloor$ are distinct, so the maximum is achieved.
\end{proof}

\section{The Integral Spectral Algebra}
\label{sec:integral}

The \emph{integral spectral algebra} $\mathcal{A}_\Z(\Cay(G,S)) = \Z[A]$
is an order (a $\Z$-lattice that is also a ring) inside
$\calA_\Q(\Cay(G,S))$.

\begin{proposition}[Integral structure]
\label{prop:integral}
\begin{sloppypar}
$\Z[A]$ is a free $\Z$-module
of rank $r = \dim_\Q \calA_\Q(\Cay(G,S))$,
with $\Z$-basis $\{I, A, A^2, \ldots, A^{r-1}\}$.
\end{sloppypar}
Its discriminant
(as an order in $\calA_\Q$) is $\mathrm{disc}(m_{G,S,\Q})$, the
discriminant of the minimal polynomial.
\end{proposition}

\begin{proof}
The set $\{I, A, \ldots, A^{r-1}\}$ is a $\Q$-basis for $\calA_\Q$
(since $\deg m = r$) and consists of integer matrices, so it spans
$\Z[A]$ as a $\Z$-module.  Independence over $\Z$ follows from
independence over $\Q$.  The discriminant formula is standard for
monogenic orders; see~\cite{Neu99}.
\end{proof}

\begin{example}[Integral algebra of $Q_2$]
$\Z[A_{Q_2}]$ has $\Z$-basis $\{I, A, A^2\}$ and minimal polynomial
$x^3 - 4x = x(x-2)(x+2)$.  The discriminant of a polynomial with
distinct roots $\lambda_1, \lambda_2, \lambda_3$ is
$\prod_{i < j}(\lambda_i - \lambda_j)^2$.  For roots $0, 2, -2$:
$\mathrm{disc}(x^3-4x) = (0-2)^2(0-(-2))^2(2-(-2))^2
= 4 \cdot 4 \cdot 16 = 256$.
The index $[\mathcal{O}_{\Q^3} : \Z[A]]$ equals $\sqrt{256/1} = 16$
(since the maximal order of $\Q^3 = \Q \times \Q \times \Q$ is $\Z^3$
with discriminant $1$, and the index squared equals $256/1$).
\end{example}

For the cyclic case, the integral spectral algebra relates to cyclotomic
orders.  The ring $\Z[2\cos(2\pi/d)]$ is the ring of integers of
$\Q(\zeta_d)^+$ (also written $\Z[\zeta_d + \zeta_d^{-1}]$), and the
integral Wedderburn decomposition is
\[
  \Z[A_{C_n}] \;\subseteq\; \prod_{d \mid n} \Z\bigl[2\cos(2\pi/d)\bigr]
  = \prod_{d \mid n} \mathcal{O}_{\Q(\zeta_d)^+},
\]
with the inclusion an equality of orders in many (but not all) cases.

\section{Non-standard Connection Sets and Character Sums}
\label{sec:nonstandard}

Most examples in this paper use the simplest connection sets
(generators and their negatives for cyclic factors, or the standard
basis for $(\Z/p)^k$).  The theory applies equally to arbitrary
symmetric $S$.

\begin{example}[Paley graph eigenvalues]
Let $p \equiv 1 \pmod{4}$ be prime and $G = \Z/p$.  The \emph{Paley
graph} $P(p)$ has connection set $S = \{a^2 : a \in (\Z/p)^*\}$ (the
quadratic residues mod $p$), which is symmetric since $p \equiv 1
\pmod{4}$ implies $-1$ is a square mod $p$, so $-S = S$.

The eigenvalues of $A_{P(p)}$ are the additive character sums
$\lambda_\chi = \sum_{s \in S} \chi(s)$ where $\chi : \Z/p \to
\overline{\Q}^*$ are the additive characters $\chi_a(g) = \omega_p^{ag}$.
The principal-character eigenvalue is $\lambda_{\chi_0} = |S| = (p-1)/2$.
For non-principal $\chi_a$ ($a \ne 0$), one writes
$\sum_{s \in QR_p} \omega_p^{as} = \frac{1}{2}\bigl(-1 + \tau(\eta,
\chi_a)\bigr)$, where $\tau(\eta, \chi_a) = \sum_{t=1}^{p-1} \eta(t)
\omega_p^{at}$ is a Gauss sum ($\eta$ the Legendre symbol).  Since
$|\tau(\eta, \chi_a)|^2 = p$, the two distinct non-principal eigenvalues
are $(-1 \pm \sqrt{p^*})/2$ where $p^* = (-1)^{(p-1)/2} p$.
The spectral algebra is therefore
$\calA_\Q(P(p)) \cong \Q \times \Q(\sqrt{p^*})$, with
$\dim_\Q = 1 + 2 = 3$ (two Galois orbits: the singleton $\{(p-1)/2\}$
and the conjugate pair $\{(-1\pm\sqrt{p^*})/2\}$) and
idempotent count $2^2 = 4$.
\end{example}

\begin{example}[Circulant with multiple generators on $\Z/9$]
Let $G = \Z/9$ and $S = \{1,-1,3,-3\}$.  Eigenvalues:
$\lambda_k = 2\cos(2\pi k/9) + 2\cos(2\pi k \cdot 3/9)
= 2\cos(2\pi k/9) + 2\cos(2\pi k/3)$.

For $k = 0$: $\lambda_0 = 4$.  For $k = 3$: $\lambda_3 = 2\cos(2\pi/3)
+ 2\cos(2\pi) = -1 + 2 = 1$.  For $k = 6$: $\lambda_6 = 2\cos(4\pi/3)
+ 2\cos(4\pi) = -1 + 2 = 1$.  For $k = 1$: $\lambda_1 = 2\cos(2\pi/9)
+ 2\cos(2\pi/3) = 2\cos(2\pi/9) - 1$.  This lies in $\Q(\zeta_9)^+$.
The distinct eigenvalues are $\{4, 1, 2\cos(2\pi/9)-1,
2\cos(4\pi/9)-1, 2\cos(8\pi/9)-1\}$, five values (the last three form
one Galois orbit over $\Q$).  Hence
$\calA_\Q(\Cay(\Z/9, \{1,-1,3,-3\})) \cong \Q^2 \times \Q(\zeta_9)^+$,
dimension $5$, idempotents $2^3 = 8$.
\end{example}

These examples illustrate that the spectral algebra is highly sensitive to
the choice of $S$: different connection sets on the same group $G$ can
give algebras of very different dimensions and Galois orbit structures.
The character-orbit decomposition (Theorem~\ref{thm:main}) provides the
unifying language for all such computations.

\section{Conclusion and Open Problems}
\label{sec:conc}

The character-orbit decomposition
\[
  \calA_K(\Cay(G,S)) \;\cong\;
  \prod_{\calO \in \calO_K(G,S)} K[x]/(\Psi_{\calO,K}(x))
\]
identifies the spectral algebra of any abelian Cayley graph as an explicit
product of field extensions of $K$, one per Galois orbit of character sums.
The proof is self-contained: DFT diagonalisation (Theorem~\ref{thm:dft})
gives the eigenvalues as character sums; the Galois action on $\Ghat$
(Lemma~\ref{lem:galois-action}) determines the orbits; separability of
$x^N - 1$ ensures the minimal polynomial is squarefree; and the Chinese
Remainder Theorem yields the Wedderburn decomposition.

The structural results cover all cases.  For $K = \Q$, every Wedderburn
summand is a real abelian number field contained in $\Q(\zeta_N)^+$.
The dimension equals the number of distinct eigenvalues over $\overline{K}$,
the idempotent count is $2^r$ where $r$ is the orbit number, and the
primitive idempotents are explicit via the Bezout algorithm.  The
Bose--Mesner perspective (Section~\ref{sec:bose}) places the spectral
algebra inside the group algebra $K[G]$ as a subalgebra generated by one
element.

Three qualitatively different regimes emerge.  The \emph{rational regime}
($p \le 3$ or all eigenvalues in $K$) gives $\calA_K \cong K^r$; this
covers Hamming cubes, grids over small primes, and many standard examples.
The \emph{quadratic regime} ($p = 5$ or similar) gives products including
$\Q(\sqrt{5})$ or other real quadratic fields.  The \emph{higher-degree
regime} ($p \ge 7$, or general $n$ with large prime divisors) produces
extensions of degree $\ge 3$, requiring the full Galois theory of real
cyclotomic fields~\cite{Was97}.

The tensor-product inequality
$\dim \calA_K(\Cay(G_1 \times G_2, S_1 \oplus S_2)) \le \dim_1 \cdot \dim_2$
(Proposition~\ref{prop:tensor}) shows that Cartesian products of Cayley
graphs have smaller spectral algebras than expected from the factors alone;
equality holds when all eigenvalue sums are distinct.  The maximum
spectral algebra on a cyclic group $\Z/N$ has dimension $\lfloor N/2\rfloor
+ 1$ and is achieved by the cycle $C_N$ (Theorem~\ref{thm:max}).

Several questions remain open.

\emph{Nonabelian groups.}  For a nonabelian group $G$ and a symmetric $S$
generating $G$, the eigenvalues of $A_{\Cay(G,S)}$ are not character sums
but arise from matrix representations; the spectral algebra becomes a direct
sum of matrix algebras.  A complete analogue of Theorem~\ref{thm:main} for
non-abelian Cayley graphs would require the full representation theory of $G$.

\emph{Cayley isomorphism problem.}  The spectral algebra type does not in
general distinguish non-isomorphic Cayley graphs on the same group
(Remark in Section~\ref{sec:invariant}).  Determining which graph properties
are detected by the spectral algebra, and which require richer invariants,
is an open problem with both algebraic and combinatorial content.

\emph{Integral orders.}  The order $\Z[A] \subseteq \calA_\Q(\Cay(G,S))$
(Section~\ref{sec:integral}) is a $\Z$-form of the spectral algebra.
Its class group, unit group, and discriminant are arithmetic invariants
connecting spectral graph theory to algebraic $K$-theory.  For the cycle
case, these invariants are related to class numbers of real cyclotomic
fields~\cite{Was97}, but the general abelian case is unexplored.

\emph{$p$-adic and local spectral algebras.}  The $p$-adic completion
$\calA_{\Q_p}(\Cay(G,S))$ decomposes according to $p$-adic Galois orbits,
which refine the $\Q$-orbits.  The interplay between the $p$-adic
factorisation of $m_{G,S,\Q}$ and the reduction modulo $p$
(Section~\ref{sec:charp}) suggests a local-global principle for spectral
algebras analogous to the Hasse--Minkowski theorem for quadratic forms.

\emph{Expansion and spectral algebra dimension.}  The spectral gap
$\gamma$ and the dimension $r$ of $\calA_K$ are related: large $r$ (many
distinct eigenvalues) tends to imply small spectral gap (eigenvalues
cluster), while Ramanujan graphs have few distinct eigenvalues.  A precise
quantitative relationship between $\gamma$, $r$, and the Galois orbit
structure of $\calA_K$ would strengthen the algebraic approach to graph
expansion.

\begin{ack}
The authors thank the referees for careful reading.

\medskip
\noindent\textbf{Competing interests.}
None.

\medskip
\noindent\textbf{Funding.}
None.

\medskip
\noindent\textbf{Data Availability.}
Data supporting this work are available from the corresponding author
(itsdeep@live.com) upon reasonable request.
\end{ack}

\appendix

\section{Extended Spectral Algebra Data}
\label{app:extended}

Table~\ref{tab:extended} gives the complete Wedderburn decomposition,
dimension, orbit count $r$, and idempotent count for additional abelian
Cayley graphs.  The notation follows Table~\ref{tab:data}: $F_d$ denotes
$\Q(\zeta_d)^+$, and $F_d = \Q$ for $d \in \{1,2,3,4,6\}$.

\begin{table}[ht]
\centering
\caption{Extended spectral algebra data. $\delta = \dim_\Q \calA_\Q$;
  $r$ = orbit count; $|$Idem$|$ $= 2^r$.}
\label{tab:extended}
\renewcommand{\arraystretch}{1.15}
\small
\begin{tabular}{p{2.5cm}ccc}
\toprule
$(G, S)$ & $\delta$ & $r$ & $|\text{Idem}|$ \\
\midrule
$C_3$, $S=\{1,-1\}$ & $2$ & $2$ & $4$ \\
$C_4$, $S=\{1,-1\}$ & $3$ & $3$ & $8$ \\
$C_5$, $S=\{1,-1\}$ & $3$ & $2$ & $4$ \\
$C_6$, $S=\{1,-1\}$ & $4$ & $4$ & $16$ \\
$C_8$, $S=\{1,-1\}$ & $5$ & $4$ & $16$ \\
$C_{10}$, $S=\{1,-1\}$ & $6$ & $4$ & $16$ \\
$C_{12}$, $S=\{1,-1\}$ & $7$ & $6$ & $64$ \\
$Q_1$ & $2$ & $2$ & $4$ \\
$Q_2$ & $3$ & $3$ & $8$ \\
$Q_3$ & $4$ & $4$ & $16$ \\
$Q_4$ & $5$ & $5$ & $32$ \\
$Q_5$ & $6$ & $6$ & $64$ \\
$\Z/6{\times}\Z/2$, $S_3$ & $7$ & $7$ & $128$ \\
$\Z/5{\times}\Z/2$, $S_3$ & $6$ & $4$ & $16$ \\
$\Z/4{\times}\Z/4$, $S_4$ & $5$ & $5$ & $32$ \\
$\Z/3{\times}\Z/3$, $S_4$ & $3$ & $3$ & $8$ \\
$\Z/2{\times}\Z/6$, $S_4$ & $6$ & $6$ & $64$ \\
$C_4\,{\square}\,C_6$ & $9$ & $9$ & $512$ \\
$C_5\,{\square}\,C_5$ & $6$ & $4$ & $16$ \\
$C_6\,{\square}\,C_6$ & $9$ & $9$ & $512$ \\
\bottomrule
\end{tabular}
\end{table}

The orbit count $r < \delta$ occurs when some Galois orbits have size
$> 1$; in such cases the Wedderburn product contains non-rational factors
and $2^r < 2^\delta$.  For the Hamming cubes all orbits are singletons
($r = \delta = k+1$); for $\Z/5 \times \Z/2$ there are two quadratic orbits
and two singleton orbits ($r = 4 < \delta = 6$).

\section{Verification of Idempotent Relations}
\label{app:idem-verify}

For completeness, we verify the idempotent relations
$e_\calO^2 = e_\calO$ and $e_\calO e_{\calO'} = 0$ in the explicit
$Q_2$ case of Example~\ref{ex:Q2-idem}.  The graph $Q_2$ has $|G| = 4$
vertices and $A$ is the $4 \times 4$ adjacency matrix.  From $A^3 = 4A$
(Cayley--Hamilton):
\begin{align*}
  e_2^2 &= \bigl(A(A+2I)/8\bigr)^2 = (A^3+2A^2)(A+2I)/64 \\
       &= (4A+2A^2)(A+2I)/64.
\end{align*}
Using $A^3 = 4A$ again: $A^2(A+2I) = A^3 + 2A^2 = 4A + 2A^2$, so
$e_2^2 = (4A+2A^2)(A+2I)/64 = A(A+2I)(4+2A-2\cdot2)/64$... the direct
check is cleaner via the projection property: since $e_2$ maps to $1$ in
the $\lambda=2$ factor and $0$ elsewhere, $e_2^2$ maps to $1^2=1$ and
$0^2=0$, confirming $e_2^2 = e_2$.  Similarly, $e_2 e_0$ maps to $1 \cdot
0 = 0$ in every factor, so $e_2 e_0 = 0$.  These projective-algebra
arguments are equivalent to direct matrix computation and confirm the
Bezout construction is exact.

Table~\ref{tab:data} records $\calA_\Q(\Cay(G,S))$ for small abelian groups
with the standard symmetric connection sets.  Here $r$ denotes the orbit
count and $F_\Delta$ denotes $\Q[x]/(x^2 - \Delta)$ (a quadratic field
when $\Delta$ is not a perfect square).

\begin{table}[ht]
\centering
\caption{Spectral algebra decomposition for selected abelian Cayley graphs.
  $N = |G|$; $\delta$ = number of distinct eigenvalues = $\dim_\Q \calA_\Q$.}
\label{tab:data}
\renewcommand{\arraystretch}{1.15}
\small
\begin{tabular}{llcc}
\toprule
$G$ & $\calA_\Q(\Cay(G,S)) \cong$ & $\delta$ & $2^r$ \\
\midrule
$\Z/3$, $S=\{1,-1\}$ & $\Q^2$ & $2$ & $4$ \\
$\Z/5$, $S=\{1,-1\}$ & $\Q \times \Q(\sqrt{5})$ & $3$ & $4$ \\
$\Z/7$, $S=\{1,-1\}$ & $\Q \times \Q(\zeta_7)^+$ & $4$ & $4$ \\
$(\Z/2)^2$, std.\ $S$ & $\Q^3$ & $3$ & $8$ \\
$(\Z/2)^3$, std.\ $S$ & $\Q^4$ & $4$ & $16$ \\
$(\Z/3)^2$, std.\ $S$ & $\Q^3$ & $3$ & $8$ \\
$(\Z/5)^2$, std.\ $S$ & $\Q^2 \times \Q(\sqrt{5})^2$ & $6$ & $16$ \\
$\Z/6{\times}\Z/2$, $S_3$ & $\Q^7$ & $7$ & $128$ \\
$\Z/5{\times}\Z/2$, $S_3$ & $\Q^2 \times \Q(\sqrt{5})^2$ & $6$ & $16$ \\
$\Z/4{\times}\Z/4$, $S_4$ & $\Q^5$ & $5$ & $32$ \\
$C_4\,\square\, C_6$ & $\Q^9$ & $9$ & $512$ \\
$C_5\,\square\, C_5$ & $\Q^2 \times \Q(\sqrt{5})^2$ & $6$ & $16$ \\
\bottomrule
\end{tabular}
\end{table}

Here $S_3$ denotes the standard connection set of size $3$
(basis vectors and their negatives for factors, plus the generator of
the $\Z/2$ factor), and $S_4$ denotes the size-$4$ standard set.
The last three rows treat Cartesian products; the dimension of
$C_5 \,\square\, C_5$ equals $6$: the three possible cosine values
$(2, \alpha, \beta)$ for each factor yield six distinct pairwise sums
$\{4, 2+\alpha, 2+\beta, 2\alpha, \alpha+\beta=-1, 2\beta\}$,
two rational singletons and two conjugate quadratic pairs.

\clearpage

\end{document}